\documentclass{article}
\usepackage[utf8]{inputenc}
\usepackage{amsfonts,amsmath,amssymb,amsthm}
\usepackage{enumerate}
\usepackage[english]{babel}
\usepackage{dsfont}
\usepackage{hyperref}
\usepackage{xcolor,bm}
\usepackage{mathtools}
\mathtoolsset{showonlyrefs}

\newtheorem{theorem}{Theorem}%meant for continuous numbers
\newtheorem{thm}{Theorem}[section]
\newtheorem{lemma}[thm]{Lemma}
\newtheorem{proposition}[thm]{Proposition}

\newtheorem{hypothesis}[theorem]{Hypothesis}
\newtheorem{Hypothesis}{Hypothesis}
\newtheorem{definition}[thm]{Definition}
\newtheorem{example}[thm]{Example}

\newtheorem{remark}[thm]{Remark}

\def\dim{\noindent \hbox{{\bf Proof.} }}
\def\R{\mathbb R}
\def\N{\mathbb N}

\def\E{\mathbb E}
\def\P{\mathbb P}
\def\Q{\mathbb Q}

\def\cald{{\cal D}}

\def\calg{{\cal G}}
\def\calh{{\cal H}}

\def\call{{\cal L}}
\def\calm{{\cal M}}

\def\calp{{\cal P}}

\def\call{{\cal L}}
\def\cals{{\cal S}}

\def\to{\rightarrow}

\def\dis{\displaystyle}

\def\<{\langle}

\def\>{\rangle}

\def\dis{\displaystyle}

\allowdisplaybreaks

\newtheorem{defn}[thm]{Definition}

\title{A Bismut -Elworthy formula for BSDEs with degenerate noise}
\usepackage{authblk}

\author[1]{Davide ADDONA\thanks{davide.addona@unipr.it}}
\author[2]{Federica MASIERO\thanks{federica.masiero@unimib.it}} 

\affil[1]{Dipartimento di Scienza Matematiche, Fisiche e Infomatiche, Universit\`a di Parma, Parma, Italy}

\affil[2]{Dipartimento di Matematica e Applicazioni, Universit\`a di Milano Bicocca, Milano, Italy}

\date{}

\begin{document}

\maketitle

\begin{abstract}
In this paper we derive a Bismut-Elworthy formula under assumptions weaker than the non degeneracy of the noise. By Bismut-Elworthy formula we mean a gradient type estimate on the transition semigroup of a stochastic differential equation in a possibly infinite dimensional Hilbert space. 
We also consider a nonlinear version of the Bismut formula for a backward stochastic differential equation, in analogy to what is done in \cite{futeBismut}, where a non-degenerate noise is considered.
Our study is motivated by applications to stochastic wave equations and to stochastic damped wave equation. 
\end{abstract}

\section{Introduction}

In this paper we investigate a Bismut-Elworthy type formula for backward stochastic differential equations (BSDEs) in a Markovian framework with possible degenerate noise.
\newline Namely, we consider a stochastic differential equation (SDE) which evolves in a real and separable Hilbert space $H$, of the form
\begin{equation}\label{sde}
\left\{\begin{array}{ll}\displaystyle  dX_t =
AX_tdt+B(t,X_t)dt+G(t)dW_t,&  t\in
[s,T]\subset [0,T],\\[1mm]
X_\tau=x, & \tau\in[0,s],
\end{array}\right.
\end{equation}
where $A:D(A)\subseteq H\to H$ is the infinitesimal generator of a strongly continuous semigroup of linear operators $(e^{tA})_{t\geq 0}$ on $H$,  $B:[0,T]\times H\rightarrow H$, $W=(W_t)_{t\geq0}$ is a cylindrical Wiener process on another real and separable Hilbert space $U$,  and $G:[0,T]\to  \mathcal L(U;H)$.
Our fundamental assumption on the diffusion term $G(t)$ is that there exists  $J\in\mathcal L(U,H)$ such that we can write $G(t)=J\widetilde G(t)$, where  $\widetilde G(t)\in \mathcal L(U)$ is invertible with bounded inverse: this is a weaker hypothesis than the non degeneracy of $G$. Besides this, on the drift $B$ we assume that $B(t,x)=J\bar B(t,x)$,  where $\bar B:[0,T]\times H\to U$.
\newline Under suitable assumptions,we are able to prove an analogous of the classical Bismut-Elworthy formula  (see the seminal papers \cite{B,ElworthyLi}), starting from the case when $B=0$. Namely we show that for every $f\in B_b(H)$, 
\begin{align}
\label{Bismut-intro}
\langle (\nabla_xP_{s,t}f)(x),h\rangle_H
={\E}\left[ f(X_t^{s,x})\int_s^t\langle \widetilde u(\sigma),dW_\sigma\rangle_{U}\right], \qquad x,h\in H, \ 0\leq s< t,
\end{align}
where $(P_{s,t})$ is the evolution operator associated to $X$ and it is given by
$$P_{s,t}f(x)=\mathbb E[f(X_t^{s,x})],\qquad  f\in B_b(H),\ x\in H, \ 0\leq s\leq t,$$ and $\widetilde u\in L^2([0,T];U)$ satisfies, for every $0\leq s<t$, 
\begin{align}
\label{formula_fond_intro}
\int_s^te^{(t-\tau)A} G(\tau)\widetilde u(\tau)d\tau= e^{(t-s)A}h, \quad \mathbb P\textup{-a.s}.
\end{align}
Condition \eqref{formula_fond_intro} is obtained by comparing the G\^ateaux derivative and the Malliavin derivative of $X$ and, in general, it is not obvious that, for every $h\in H$, there exists $\widetilde u$ satisfying \eqref{formula_fond_intro}. If $G$ is invertible (and also depends on $x$), then it is possible to show that such a $\widetilde u$ exists for every $h\in H$, but in this case the construction of $\widetilde u$ explicitly depends on the inverse of $G$, and so it is not easily extendable to the degenerate case. The same situation emerges assuming that $GG^*$ is invertible; indeed, it is possible to prove the existence of $\widetilde u$, see for instance \cite{Wang}, but also in this case the inverse of $GG^*$ explicitly appears in $\widetilde u$. Since in our examples neither $G$ nor $GG^*$ are invertible, these results are not applicable.

%see Section \ref{subsez mall-gradient} for details and precise formulation. 
Formula \eqref{Bismut-intro} leads to the classical estimate $|\langle(\nabla_x P_{s,t}f)(x),h\rangle_H| 
\leq C\|f\|_\infty\|h\|_H$, where $C$ is a positive constant which does not depends either on $f$ or $h$. We refer to \eqref{Bismut-intro} as linear Bismut formula since $u(s,x)= \mathbb E[\phi(X_T^{s,x})]=P_{s,T}[\phi](x)$ is the solution of the linear Kolmogorov equation\begin{equation*}
\left\{
\begin{array}
[c]{l}%
-\frac{\partial v}{\partial s}(\tau,x)=\mathcal L_s v\left(  \tau,x\right),\quad x\in H,\  s\in\left[  0,T\right],\\[2mm]
v(T,x)=\phi\left(  x\right)  , \quad x\in H,
\end{array}
\right.  
\end{equation*} 
%\begin{align*}
% (\call_s f)(x)&=\frac{1}{2}(Tr G(s,x)G^*(s,x) \nabla^2 f)(x)+\langle Ax,\nabla f(x)\rangle\\&\,+\langle B(s,x),\nabla f(x)\rangle.
%\end{equation*}
where $(\mathcal{L}_\tau)$ is at least formally the generator of $(P_{s,\cdot})$
\newline If we turn to semilinear Kolmogorov equations in $H$ of the form
\begin{equation}
\left\{
\begin{array}
[c]{l}%
-\frac{\partial v}{\partial s}(s,x)=\mathcal L_s v\left(  s,x\right)
+\psi\left( s,x,v(s,x),\nabla v(s,x)G(s)   \right), \quad x\in H, \ s\in\left[  0,T\right],
\\[2mm]
v(T,x)=\phi\left(  x\right)  , \qquad h\in H,
\end{array}
\right.  \label{Kolmointro}
\end{equation} 
then it is well known that with $\psi$ Lipschitz continuous in $v$ and in $\nabla v G$, there exists a unique mild solution 
$$u(s,x)= \mathbb E[\phi(X_T^{s,x})+\int_s^T\psi(r,X_r^{s,x},Y_r^{s,x},Z_r^{s,x})\,dr]$$
to \eqref{Kolmointro}. Such a solution $u$ can be defined by means of a system of forward-backward stochastic differential equations (FBSDEs). Namely, we consider a BSDE coupled with an SDE
\begin{equation*}
\left\{\begin{array}{ll}\displaystyle  dX_t =
AX_tdt+B(t,X_t)dt+G(t)dW_t,&  t\in[s,T]\subset [0,T],\\[1mm]
X_s=x, & \tau\in[0,s], \\[1mm]
-dY_t =\psi(t,X_t,Y_t,Z_t)\;dt-Z_t\;dW_t,
& t\in[0,T], \\[1mm]
Y_T=\phi(X_T),
\end{array}\right.
\end{equation*}
and it turns out that the solution of equation \eqref{Kolmointro} can be defined as $u(s,x)=Y^{s,x}_s$. In the present paper we are able to prove a Bismut formula for BSDEs (which we refer to as the \textit{semilinear Bismut formula} because of the connection between BSDEs and semilinear Kolmogorov equations)
\[
\mathbb E[
\langle \nabla_x Y_t^{s,x}, h\rangle_H] = \mathbb{E} \left[ \int_t^T \psi(r, X_r^{s,x}, Y_r^{s,x}, Z_r^{s,x})  \int_s^r\langle\widetilde u(\sigma),dW_\sigma\rangle_U \,dr+ \phi(X_T) \int_s^T\langle \widetilde u(\sigma),\,dW_\sigma\rangle_U \right]
\]
for every $0\leq s<t$ and every $h\in H$ such that \eqref{formula_fond_intro} holds. 
This formula is an extension, to the case of non invertible diffusion $G$ (or $GG^*$), of the semilinear Bismut formula studied in \cite{futeBismut} (or in \cite{Wang}) in the case of invertible diffusion; we consider the case of degenerate diffusion since we are
motivated by the regularizing properties of the transition semigroup of the stochastic wave equations, studied in \cite{MaPr24}, and of the  stochastic damped wave equation, first studied in \cite{AddMasPri23} and later in \cite{AddBig24,AddBig26}. Indeed, the stochastic (damped) wave equation can be reformulated as a stochastic evolution equation with degenerate noise, see \cite{AddBig24,AddBig26,AddMasPri23,MaPr,MaPr24}. To ensure maximum generality, we will present our results in an abstract form and then we apply them to the motivating models.
	
\medskip

Gradient derivative formulae for Wiener functionals have been introduced  in \cite{B} and have been extended in \cite{ElworthyLi}. We also refer to \cite{DP3} where the Bismut-Elworthy-Li formula is presented in the infinite dimensional setting, and \cite{Ce} where it has been studied for reaction-diffusion equations.
\newline In \cite{futeBismut} semilinear version of the Bismut-Elworthy formula is considered, with its applications to semilinear Kolmogorov equations and to stochastic optimal control problems; in \cite{Masiero} and in \cite{AddBandMas} ( in this last paper the Banach space setting is considered) a semilinear Bismut formula is considered when the generator $\psi$ in the BSDE is quadratic with respect to $Z$. 

We mention also  papers concerning the Bismut-Elworthy-Li formulae for stochastic differential equations (SDEs) with jumps, among others \cite{RenZhang}  and \cite{Ta} for the case of Poisson jump processes and Poisson random measures, and \cite{Zhang} for
SDEs driven by $\alpha$ stable processes
Moreover in \cite{RenZhang-arxiv} the semilinear version of the  Bismut-Elworthy-Li formulae for BSDEs with Poisson random measure is studied.

In all the aforementioned literature, the problem is studied under the assumption of non-degeneracy of the noise. As previously noticed, in \cite{Wang} the author weakens the assumption on the diffusion coefficients, requiring a sort of non degeneracy of the covariance operator. However, this approach does not cover the case of stochastic (damped) wave equations which can be reformulated as abstract evolution equations with highly non degenerate noise. Roughly speaking, such equations can be reformulated in a product space for the pair position-velocity and the noise affects only the evolution of the equation for the state, and this gives rise to the degeneracy in the noise coefficient. In \cite{FanRoZha} the authors consider Bismut type formulas for 
BSDEs,  and they are able to consider a wider class of  forward equations than \cite{futeBismut} including also McKean-Vlasov equations, but they don't cover the case of stochastic (damped) wave equations.  We also mention the recent paper \cite{ScaZa} where the authors study a Bismut formula for particular stochastic evolution equations arising in thermodynamics and phase separation models in which the noise has special non degeneracy.

The paper is organized as follows. In Section \ref{Sec:2} we state the main assumptions on the coefficients of the forward equation \eqref{sde}, recall the main ingredients of Malliavin calculus, identify the G\^ateaux and the Malliavin derivative of $X$ and show that it leads to \eqref{formula_fond_intro}. Finally, we prove the linear Bismut formula \eqref{Bismut-intro} and its estimate, which also holds for functions $f$ with polynomial growth. In Section \ref{sec:fbsde}, we introduce the system of forward-backward differential equations and prove the semilinear Bismut formula for the process $Y$. As usual, at first we prove such a result for smooth coefficients and $B=0$, and then we address the general case by Girsanov's theorem and a classical approximating procedure. In Section \ref{sec:pde}, we apply the results obtained in the previous section to identify $Z$ with the directional gradient of $Y$ and to provide existence and uniqueness to the semilinear Kolmogorov equation \eqref{Kolmointro}. Finally, in Sections \ref{sec:exa_wave} and \ref{sec:exa_damped} we show that our abstract results apply to a class of nonlinear stochastic wave equations and of stochastic damped wave equations, respectively.

\section{A general abstract setting-with drift}
\label{Sec:2}
Let us consider a complete probability space $(\Omega, \mathcal F,\mathbb P)$, and let $H$ be a separable Hilbert space. For every $T>0$, every $s\in[0,T)$ and every $x\in H$, we consider the SPDE, evolving in $H$, given by
\begin{align}
\label{eq_nonlin_1}
dX_t^{s,x}=AX_t^{s,x}dt+B(t,X_t^{s,x})dt+G(t)dW_t, \quad t\in[s,T],\  X_s^{s,x}=x \in H,
\end{align}
where $W=(W_t)_{t\geq0}$ is a cylindrical Wiener process on a separable Hilbert space $U$, $A:D(A)\subseteq H\rightarrow H$ generates a strongly continuous semigroup $(e^{tA})_{t\geq0}$, $B:[0,T]\times H\rightarrow H$ and $G:[0,T]\to  \mathcal L(U;H)$.
%belongs to $\mathcal G^{1}(H;\mathcal L(U;H))$, $G$ has linear growth and it is (locally) Lipschitz continuous. 
We consider the natural filtration $(\mathcal F_t)_{t\in[0,T]}$ of $W$, augmented with the $\mathbb P$-null sets of $\mathcal F_T$, which satisfies the usual conditions.

We assume the following hypotheses on the coefficients of \eqref{eq_nonlin_1}.
\begin{hypothesis}
\label{hyp:base}
\begin{itemize}
\item[\rm(i)]
There exist a measurable and bounded function $\bar B:[0,T]\times H\to U$ and $J\in\mathcal L(U;H)$ such that $B(t,x)=J\bar B(t,x)$ for every $(t,x)\in[0,T]\times H$. Further, 
%The drift $B:[0,T]\times H\to H$ is measurable and \textcolor{blue}{bounded} and 
there exist positive constants $C_B$ and $\gamma_B\in[0,1)$ such that 
%\textcolor{red}{ho tolto la crescita lineare di $B$ che chiediamo la limitatezza}
%usiamo mai $\gamma_B$?Fe: no maio lo lascerei per maggiore generalità Da: ok, mettiamo B invece di $\bar B$? semigruppo applicato a }
\begin{align}
%& \|e^{tA}\bar B(\tau,x)\|\leq C_B t^{-\gamma_B}(1+|x|_H), \label{hyp:base1B}\\
& \|e^{tA} B(\tau,x)-e^{tA} B(\tau,y)\|\leq C_B t^{-\gamma_B}|x-y|_H, \label{hyp:base2B}
\end{align}
for every $t\in(0,T]$, every $\tau\in[0,T]$ and every $x,y\in H$.
Moreover, $e^{tA}B(\tau,\cdot)\in\mathcal G^1(H;H)$ for every $t\in(0,T]$ and every $\tau\in[0,T]$. 
\item[\rm(ii)] 
%The diffusion $G:[0,T]\times H\to \mathcal L(U;H)$ is such that for every $u\in U$ the mapping $G(\cdot,\cdot,\cdot)u:[0,T]\times H\to H$ is measurable, $e^{tA}G(\tau,x)\in\mathcal L_2(U;H)$ for every $t\in(0,T]$, every $\tau\in[0,T]$ and every $x\in H$ and there exist positive constants $C_G$ and $\gamma_G\in\left(0,\frac12\right)$ such that
%\begin{align}
%& \|e^{tA}G(\tau,x)\|_{\mathcal L_2(U;H)}\leq C_G t^{-\gamma_G}(1+|x|_H), \label{hyp_base1}\\
%& \|e^{tA}G(\tau,x)-e^{tA}G(\tau,y)\|_{\mathcal L_2(U;H)}\leq C_G t^{-\gamma_G}|x-y|_H, \label{hyp_base2} \\
%& \|G(\tau,x)\|_{\mathcal L(U;H)}\leq C_G(1+|x|_H)    \label{hyp_base3}
%\end{align}
%for every $t\in(0,T]$, every $\tau\in[0,T]$ and every $x,y\in H$.
The diffusion $G$ is such that $e^{tA}G(\tau)\in
\mathcal L_2(U;H)$ 
for every $t\in(0,T]$ and $\tau\in[0,T]$ and there exist positive constants $C_G$ and $\gamma_G\in\left(0,\frac12\right)$
 \begin{equation}
\|e^{tA}G(\tau)\|_{\mathcal L_2(U;H)}\leq C_G t^{-\gamma_G}, \qquad t\in(0,T], \ \tau\in[0,T]
.\label{hyp_base1}
\end{equation}
\end{itemize}
\end{hypothesis}
Under the above assumptions, from \cite[Proposition 3.2]{FuhTes2002} it follows that for every $p\in[2,\infty)$ equation \eqref{eq_nonlin_1} admits a unique mild solution $X^{s,x}\in L^p_{\mathcal P}(\Omega;C([s,T];H))$, where $X^{s,x}=(X_t^{s,x})_{t\in[s,T]}$ is a predictable process with continuous paths in $H$, satisfying, $\mathbb P$-a.s.,
\begin{align}
\label{mild_sol_nonlineare_nodrift}
X_t^{s,x}=e^{(t-s)A}x+\int_s^te^{(t-r)A} B(r,X_r^{s,x})dr+\int_s^te^{(t-r)A} G(r)dW_r, \qquad \forall t\in[s,T],
\end{align}
and
\begin{align}
\label{stima_sup_X}
\mathbb E\sup_{\sigma\in[s,T]}|X^{s,x}_\sigma|^p\leq C(1+|x|)^p,    
\end{align}
where $C>0$ is a constant which depends on $p$, $C_B$, $C_G$, $\gamma_B$, $\gamma_G$, $T$ and $M:=\sup_{t\in [0,T]}\|e^{tA}\|_{\mathcal L(H)}$. \\
We extend the process $X^{s,x}$ in $[0,s]$ by setting $X_t^{s,x}=x$ for every $t\in[0,s]$ and we still denote such an extension by $X^{s,x}$. From \cite[Proposition 3.3]{FuhTes2002} it follows that for every $p\in[2,\infty)$ the mapping $(s,x)\mapsto X^{s,x}$ belongs to $\mathcal G^{0,1}([0,T]\times H;L^p_{\mathcal P}(\Omega;C([0,T];H)))$ and the process $\nabla_xX^{s,x}h$ is the unique mild solution to the first variation equation
\begin{align*}
d\Xi_t=A\Xi_tdt+\nabla_xB(t,X_t^{s,x})\Xi_t^{s,x}dt,
%+\nabla_xG(t,X_t^{s,x})Z_t^{s,x}dW_t,
\quad \forall t\in[s,T], \ \Xi_{t}=h, \quad \forall t\in[0,s],
\end{align*}
i.e., $\nabla_xX^{s,x}h=(\nabla_x X_t^{s,x}h)_{t\in[0,T]}$ satisfies, $\mathbb P$-a.s.,
\begin{align}
\label{mild_solu_first_var}
\left\{
\begin{array}{ll}
\nabla_x X_t^{s,x}h
=  e^{(t-s)A}h+\displaystyle \int_s^te^{(t-\sigma)A}\nabla_x B(\sigma,X_\sigma^{s,x})\nabla_xX_\sigma^{s,x}hd\sigma,  & \forall t\in[s,T], \\
%\qquad\qquad \quad +\displaystyle \int_s^te^{(t-\sigma)A}\nabla_x G(\sigma,X_\sigma^{s,x})\nabla_xX_\sigma^{s,x}hdW_\sigma, &  \forall t\in[s,T], \vspace{1mm} \\
\nabla_x X_t^{s,x}h = h, & \forall t\in[0,s).
\end{array}
\right. 
\end{align}
\subsection{Malliavin calculus}
% Let us consider the space $K=L^2(0,T;U)$ and
Let $K$ be a separable Hilbert space and let $W:K\rightarrow L^2(\Omega,\mathcal F,\mathbb P)$ be an isonormal Gaussian process, i.e., $W(k)$ is a real Gaussian random variable for every $k\in K$ and
\begin{align*}
\mathbb E[W(k_1)W(k_2)]=\langle k_1,k_2\rangle_K, \qquad k_1,k_2\in K.    
\end{align*}
Assume that $\mathcal F$ is generated by $\{W(k):k\in K\}$. It is well-known that the space $\mathcal S$ of functions $F:\Omega\rightarrow \mathbb R$ such that $F=\varphi(W(k_1),\ldots,W(k_n)) $ with $n\in\mathbb N$, $k_1,\ldots,k_n\in K$ and $\varphi\in C^\infty(\R^n)$ has at most polynomial growth, together with its derivatives, is dense in $L^p(\Omega,\mathcal F,\mathbb P)$ for every $p\in[1,\infty)$. 
Further, given a separable Hilbert space $V$, we denote by $\mathcal S(V)$ the vector space generated by functions of the form $F\otimes v$, where $F\in \mathcal S$ and $v\in V$. It is not hard to prove that $\mathcal S(V)$ is dense in $L^p(\Omega,\mathcal F,\mathbb P;V)$ for every $p\in[1,\infty)$. If $V=\R$ then we write $\mathcal S$ instead of $\mathcal S(\R)$.

We introduce the operator $D:\mathcal S(V)\rightarrow L^p(\Omega,\mathcal F,\mathbb P;\mathcal L_2(K;V))$ defined as
\begin{align*}
DF=\sum_{\ell=1}^n\sum_{r=1}^m\frac{\partial \varphi_r}{\partial\xi_\ell}(W(k_1),\ldots,W(k_n))k_\ell\otimes v_r
\end{align*}
on functions $F$ of the form
\begin{align}
\label{vect_smo_fct}
F=\sum_{r=1}^m\varphi_r(W(k_1),\ldots,W(k_n))v_r\in\mathcal S(V),
\end{align}
$n,m\in \mathbb N$, $k_1,\ldots,k_n\in K$, $v_1,\ldots,v_m\in V$ and $\varphi_r\in S$ for every $r=1,\ldots,m$. The operator $D$ is closable in $L^p(\Omega,\mathcal F,\mathbb P; V)$ for every $p\in[1,\infty)$, we still denote by $D$ its closure and we denote by $\mathbb D^{1,p}(V)$ the domain of its closure.

We denote by $\delta:{\rm dom}(\delta)\subseteq L^2(\Omega,\mathcal F,\mathbb P;\mathcal L_2(K;V))\to L^2(\Omega,\mathcal F,\mathbb P;V)$ the adjoint operator of $D$ in $L^2(\Omega, \mathcal F,\mathbb P;V)$, i.e., 
\begin{align*}
{\rm dom}(\delta) & :=
\left\{
\begin{array}{l}
F\in L^2(\Omega,\mathcal F,\mathbb P;\mathcal L_2(K;V)):\exists M\in L^2(\Omega,\mathcal F,\mathbb 
P;V)\textrm{ such that}\vspace{1mm}\\
\displaystyle \int_\Omega \langle F,D\psi\rangle_{\mathcal L_2(K;V)} d\mathbb P=\int_{\Omega} \langle M,\psi\rangle_{V} d\mathbb P, \ \ \ \forall \psi\in \mathbb D^{1,2}(V)
\end{array}
\right\}, \\
\delta F & :=M, \qquad F\in {\rm dom}(\delta).
\end{align*}
Since $D$ is a closed and densely defined operator in $L^2(\Omega,\mathcal F,\mathbb P;V)$, it follows that $\delta$ is a closed and densely defined operator in $L^2(\Omega,\mathcal F,\mathbb P;\mathcal L_2(K;V))$.

% Let us consider the adjoint operator $\delta:{\rm dom}(\delta)\rightarrow L^2(\Omega,\mathcal F,\mathbb P;V)$ of $D$ in $L^2(\Omega,\mathcal F,\mathbb P;V)$. From ...(Bogachev) it follows that $\mathbb D^{1,2}(L_2(K;V))\subseteq {\rm dom}(\delta)$ and $\|\delta u\|_{L^2}\leq \|u\|_{\mathbb D^{1,2}}$.

In the particular situation $K=L^2(0,T;U)$, $V=H$ and $W(k):=\int_0^T\langle k(t), dW_t\rangle _{U}$, where $W$ is the $U$-valued cylindrical Wiener process introduced above, we infer that $DW(k)=k$ for every $k\in K$. In general, for every $F\in \mathbb D^{1,2}(V)$ we get $DF\in L^2(\Omega;\mathcal L_2(K;H))\sim L^2(\Omega\times [0,T];\mathcal L_2(U;H))$. Indeed, if $F\in\mathcal S(H)$ is as in \eqref{vect_smo_fct} and we assume, without loss of generality, that the elements of $\{v_1,\ldots,v_m\}$ are orthonormal in $H$ and the elements of $\{k_1,\ldots,k_n\}$ are orthonormal in $K$, then
\begin{align*}
\|DF\|_{L^2(\Omega;\mathcal L_2(K;H))}^2
= &\mathbb E\sum_{j=1}^\infty\left|\int_0^T\langle DF(s),k_j(s)\rangle_Uds\right|_H^2
= \mathbb E\sum_{j=1}^n\left|\sum_{r=1}^m\frac{\partial \varphi_r}{\partial \xi_j}v_r\int_0^T|k_j(s)|^2_Uds\right|_H^2 \\
= & \sum_{j=1}^n\sum_{r=1}^m\mathbb E\left|\frac{\partial \varphi_r}{\partial \xi_j}\right|^2,
\end{align*}
and, for every orthonormal basis $\{u_\ell:\ell\in \N\}$ of $U$,
\begin{align*}
\|DF\|_{L^2(\Omega\times [0,T];\mathcal L_2(U;H))}^2
= & \mathbb E\int_0^T\sum_{\ell=1}^\infty|\langle DF(s),u_\ell\rangle _U|_H^2ds
= \mathbb E  \int_0^T\sum_{\ell=1}^\infty\left|\sum_{j=1}^n\sum_{r=1}^m\frac{\partial \varphi_r}{\partial \xi_j}v_r\langle k_j(s),u_\ell\rangle_U\right|_H^2ds\\
= & \sum_{r=1}^m\mathbb E \int_0^T\sum_{\ell=1}^\infty\left(\sum_{j=1}^n\frac{\partial \varphi_r}{\partial \xi_j}\langle k_j(s),u_\ell\rangle_U\right)^2ds \\
= & \sum_{r=1}^m\sum_{j,h=1}^n\mathbb E\left|\frac{\partial \varphi_r}{\partial \xi_j}\right|\left|\frac{\partial \varphi_r}{\partial \xi_h}\right|\int_0^T\sum_{\ell=1}^\infty\langle k_j(s),u_\ell\rangle_U\langle k_h(s),u_\ell\rangle_Uds
=  \sum_{r=1}^m\sum_{j=1}^n\mathbb E\left|\frac{\partial \varphi_r}{\partial \xi_j}\right|^2.
\end{align*}
By density, we infer that $\|DF\|_{L^2(\Omega;\mathcal L_2(L^2(0,T;U);H))}=\|DF\|_{L^2(\Omega\times [0,T];\mathcal L_2(U;H))}$ for every $F\in \mathbb D^{1,2}(H)$. We set $D_sF(\omega):=DF(\omega)(s)=DF(\omega,s)$ for every $(\omega,s)\in\Omega\times [0,T]$.

Finally, we introduce the set $\mathbb L^{1,2}(H)$, whose elements are those processes $u\in L^2(\Omega\times [0,T];H)$ such that $u_r\in \mathbb D^{1,2}(H)$ for almost every $r\in[0,T]$ and there exists a measurable version of $D_su_r$ such that
\begin{align*}
\|u\|_{\mathbb L^{1,2}(H)}
=\|u\|_{L^2(\Omega\times [0,T];H)}^2+\mathbb E\int_0^T\int_0^T\|D_su_r\|^2_{\mathcal L_2(U;H)}drds<\infty.
\end{align*}

From \cite[Proposition 3.5]{FuhTes2002} it follows that the solution $X^{s,x}$ of equation \eqref{eq_nonlin_1} belongs to $\in \mathbb L^{1,2}(H)$ and there exists a version of $DX^{s,x}$ such that for every $\tau\in[0,T)$ the process $(D_\tau X_t^{s,x})_{t\in(\tau,T]}$ is predictable with values in $\mathcal L_2(U;H)$. Moreover, $X_t^{s,x}\in \mathbb D^{1,2}(H)$ for every $t\in[0,T]$ and, $\mathbb P$-a.s.,
\begin{align*}
D_\tau X^{s,x}_t
& =  \int_s^te^{(t-r)A}\nabla_x B(r,X^{s,x}_r)D_\tau X^{s,x}_r dr 
+\mathds1_{[s,t]}(\tau)e^{(t-\tau)A}G(\tau),  && 0\leq \tau <t\leq T, \\
D_\tau X_t^{s,x} & =0, && 0\leq  t\leq \tau <T.
\end{align*}

Let us fix $s\in[0,T]$ and consider $s<t\leq T$. Then, for every $\widetilde u\in L^2(\Omega\times [0,T];U)$ we get, $\mathbb P$-a.s.,
\begin{align}
\notag
\langle DX^{s,x}_t,\widetilde u\rangle_{L^2(0,T;U)}
= & \int_0^TD_\tau X_t^{s,x} \widetilde u(\tau) d\tau \\
= & \int_s^te^{(t-r)A}\nabla_x  B(r,X^{s,x}_r)\langle DX^{s,x}_r,\widetilde u\rangle_{L^2(0,T;U)}dr 
%\\&+\int_s^te^{(t-r)A}\nabla_x  G(r,X^{s,x}_r)\langle DX^{s,x}_r,\widetilde u\rangle_{L^2(0,T;U)}dW_r  
+\int_s^te^{(t-\tau)A} G(\tau)\widetilde u(\tau)dr.
\label{mild_sol_mall_der}
\end{align}

\subsection{Identification of gradient and Malliavin derivative}
\label{subsec:mall-gradient}

From now on we assume that $B=0$, and we will come back later to the case of $B$ not identically equal to zero.
\newline Let us fix $h\in H$ and compare \eqref{mild_solu_first_var} and \eqref{mild_sol_mall_der} when $B=0$. If we find a family $(\widetilde u_t)_{t\in(s,T]}\subseteq L^2([0,T];U)$ such that for every $t\in(s,T]$ we get
\begin{align}
\label{link_gat_der_mall_der}
\int_s^te^{(t-\tau)A} G(\tau)\widetilde u_t(\tau)d\tau= e^{(t-s)A}h, \quad \mathbb P\textup{-a.s},
\end{align}
then we deduce that $\nabla_x X^{s,x}_th=\langle DX^{s,x}_t,\widetilde u_t\rangle_{L^2(0,T;U)}$ $\mathbb P$-a.s. for every $t\in(s,T]$. We stress that the family $(\widetilde u_t)_{t\in(s,T]}$ also depends on $s,x$ and $h$.
\newline Let us introduce the semigroup $(P_t)_{t\geq0}\in\mathcal L(B_b(H))$ defined as $(P_tf)(x)=\mathbb E[f(X_t^x)]$ for every $f\in B_b(H)$, every $x\in H$ and every $t\in[0,T]$, where $X^x=(X_t^x)_{t\in[0,T]}$ denotes the unique mild solution to \eqref{mild_sol_nonlineare_nodrift} with $s=0$ and $B=0$. Let $h\in H$ be such that there exists $(\widetilde u_t)_{t\in (0,T]} \subseteq L^2([0,T];U)$ such that \eqref{link_gat_der_mall_der} holds true. Then, $P_tf\in C_b^1(H)$ for every $f\in C^1_b(H)$ and $t\in(0,T]$ and we get
\begin{align}
\langle (\nabla_xP_tf)(x),h\rangle_H
= & \mathbb E[\nabla f(X_t^x)\nabla_xX_t^xh]
= \mathbb E[\nabla f(X_t^x)\langle DX_t^x,\widetilde u_t\rangle_{L^2(0,T;U)}] \notag \\
= & \mathbb E[\langle D(f(X_t^x)),\widetilde u_t\rangle_{L^2(0,T;U)}], \qquad t\in(0,T],
\label{bismut_astr_1}
\end{align}
where we have exploited the fact that $f(X_t^x)\in \mathbb D^{1,2}(H)$ and $D(f(X_t^x))=\nabla f(X_t^x)DX_t^x$. Assume that $\widetilde u_t\in {\rm dom}(\delta)$. Then, integrating by parts in \eqref{bismut_astr_1} it follows that
\begin{align}
\label{int_parit_derivata}
\langle (\nabla_xP_tf)(x),h\rangle_H
= \mathbb E[f(X_t^x)\delta(\widetilde u_t)], \qquad t\in(0,T].
\end{align}
Therefore, if we consider $f\in B_b(H)$ and we pointwise approximate it with a uniformly bounded sequence $(f_n)_{n\in\N}\subseteq C_b^1(H)$, then from \eqref{int_parit_derivata} we deduce that $P_tf$ is differentiable along $h$ and 
\begin{align*}
\langle (\nabla_x P_tf)(x),h\rangle_H=\mathbb E[f(X_t^x)\delta(\widetilde u_t)], \qquad t\in(0,T].
\end{align*}
Finally, for every $t\in(0,T]$ we get
\begin{align}
\label{stima_der_semigruppo}
|\langle (\nabla_x P_tf)(x),h\rangle_H|
=& |\mathbb E[f(X_t^x)\delta(\widetilde u_t)]|
\leq \|f\|_\infty \mathbb  E[|\delta(\widetilde u_t)|]
\leq  \|f\|_\infty (\mathbb  E[|\delta(\widetilde u_t)|^2])^{\frac12}.
\end{align}
We recall that $L^2_{\mathcal P}(\Omega,\mathcal F,\mathbb P;K)\subseteq {\rm dom}(\delta)$, where $L^2_{\mathcal P}(\Omega,\mathcal F,\mathbb P;K)$ is the space of $K$-valued $L^2$-predictable processes,  $\delta(v)=\int_0^T\langle v(t),dW_t\rangle_{U}$ for every $v\in L^2_{\mathcal P}(\Omega,\mathcal F,\mathbb P;K)$ and
\begin{align}
\label{stima_L2_div}
\mathbb E[|\delta(v)|^2]^{\frac12}=\mathbb E[\|v\|_{L^2(\Omega,\mathcal F,\mathbb P;K)}^2]^{\frac12}, \qquad v\in L^2_{\mathcal P}(\Omega,\mathcal F,\mathbb P;K).
\end{align}
Hence, since $\widetilde u_t\in K\subseteq L^2_{\mathcal P}(\Omega,\mathcal F,\mathbb P;K)$ for every $t\in(0,T]$, then from \eqref{stima_der_semigruppo} and \eqref{stima_L2_div} we infer that
\begin{align*}
  |\langle (\nabla_x P_tf)(x),h\rangle_H|
\leq \|f\|_\infty \mathbb \|\widetilde u_t\|_{L^2([0,T];U)}.
\end{align*}
Again, we recall that the family $(\widetilde u_t)_{t\in(0,T]}$ also depends on $x$ and $h$.

We introduce the space $C_K(H)$ of continuous functions with polynomial growth with respect to $x\in H$, that is there exists $K\in \N$ such that the functions
\begin{equation*}\label{eq:polgrowth}
H\ni x\mapsto \dfrac{f(x)}{(1+\vert x\vert^2) ^{K/2}},
\end{equation*}
is bounded and we endow $C_K(H)$ with the norm
\begin{equation*}
\Vert f \Vert_{C_K} =\sup_{x\in H}  \dfrac{f(x)}{(1+\vert x\vert^2) ^{K/2}},
\end{equation*}
so that $C_K(H)$ turns ut to be a Banach space. Note that if $K=0$, $C_0(H)=C_b(H)$ and $\Vert \cdot \Vert_{C_0}=\Vert \cdot \Vert_\infty$.

So if we repeat the previous calculations, starting from $f\in C_K(H)$ and taking \eqref{stima_sup_X} into account, then we get 
\begin{align*}
\label{stima_der_semigruppo_poli}
|\langle (\nabla_x P_tf)(x),h\rangle_H|
=& |\mathbb E[f(X_t^x)\delta(\widetilde u_t)]|
\leq \|f\|_{C_K} (1+\vert x\vert^2) ^{K/2})\mathbb  E[|\delta(\widetilde u_t)|] \notag \\
\leq &  \|f\|_{C_K} (1+\vert x\vert^2) ^{K/2}) (\mathbb  E[|\delta(\widetilde u_t)|^2])^{\frac12},
\end{align*}
where $\widetilde u$ satisfies \eqref{link_gat_der_mall_der}, and, in analogy to what we have done before, we end up with
\begin{align*}
\|\langle \nabla_x P_tf,h\rangle\|_{C_K} 
\leq \|f\|_{C_K} \mathbb \|\widetilde u_t\|_{L^2([0,T];U)}.
\end{align*}

% \subsection{The second order derivative}
% Let us fix $t>0$, $f\in C^1_b(H)$ and $h,k\in H$, and assume that . Then, we have
% \begin{align*}
% \langle (\nabla^2P_tf)(x)h,k\rangle_H
% = & \langle \nabla(\langle (\nabla P_t)f,h\rangle_H)(x),k\rangle_H \\
% = & \langle \nabla\mathbb E[f(X_t^x)\delta(\widetilde u)],k\rangle_H
% \end{align*}

\begin{remark}\label{rm-final} It turns out that in order to obtain first order differentiability of $P_tf$ for continuous functions $f$, it is crucial to prove that \eqref{link_gat_der_mall_der} holds true.
\newline In the literature it is well-known that when $G(t)$ is invertible \eqref{link_gat_der_mall_der} holds true and it is possible to prove the so called Bismut formula, see \cite{B}, \cite{DP3} and the references therein, see also \cite{Ce} where stochastic reaction diffusion equations and Schauder estimates for there transition semigroups are investigated by means of Bismut formulae. In the next Sections we consider stochastic evolution equations with diffusion coefficient possibly degenerate, and we apply our results to the stochastic wave equation and to the stochastic damped wave equation: we are able to prove the analogous of the Bismut formula even if these equations admit an abstract riformulation with $G(t)$ not invertible.
\end{remark}

\begin{remark}
%\label{remark:G-const}
In the present paper we assume that $G$ does not depend on $x$. We are aware that in finite dimension  a similar estimates is proved in \cite{CriDel} for operators $G$ depending also on $x$. The extension of this estimate to infinite dimensions is an open problem up to our knowledge. 
\end{remark}

\section{Bismut-Elworthy formula for forward-backward system}
\label{sec:fbsde}

In this section we consider the following forward-backward system of stochastic differential equations, where the forward equation is given by \eqref{eq_nonlin_1}. So for given $s\in [0,T]$ and $x\in H$, we consider 
\begin{equation}\label{fbsde}
\left\{\begin{array}{ll}\dis  dX_t =
AX_tdt+B(t,X_t)dt+G(t)dW_t,&  t\in
[s,T]\subset [0,T],\\[1mm]
X_\tau=x, & \tau\in[0,s], \\[1mm]
-dY_t =\psi(t,X_t,Y_t,Z_t)\;dt-Z_t\;dW_t,
 & t\in[0,T], \\[1mm]
  Y_T=\phi(X_T),
\end{array}\right.
\end{equation}
for the unknown $(X,Y,Z)$, also denoted by
$(X^{s,x},Y^{s,x},Z^{s,x})$ to stress the dependence on the initial conditions $s$ and $x$. 
%The process $X$, which is solution of the forward equation (\ref{forward}), has been extended for $0\leq s\leq t$ by setting $X_s=x$ for $0\leq s\leq t$. 
The second equation is of backward type for the unknown $(Y,Z)$,
it depends on the process $X$ and we will refer to it as backward stochastic differential equation (BSDE) in a Markovian framework in the following.
Under suitable assumptions on the coefficients
 $\psi:[0,T]\times H\times \R\times U
\rightarrow\mathbb{R}$
and $\mathbb{\phi}:H\rightarrow\mathbb{R}$, we will look for a solution consisting of a pair of predictable processes,
taking values in $\mathbb{R}\times U$, such that $Y$ has
continuous paths and
\[
\|\left( Y,Z\right)\|^2_{\cals^2\times\calm^2}:=
\mathbb{E}\sup_{\tau\in\left[ 0,T\right] }\left\vert Y_{\tau}\right\vert
^{2}+\mathbb{E}\int_{0}^{T}\left\vert Z_{\tau}\right\vert ^{2}d\tau<\infty,
\]
see e.g. \cite{PaPe1} for the case where the generator $\psi$ is assumed to be Lipschitz continuous with respect to $Y$ and $Z$, as stated in the following assumptions.
\begin{hypothesis}
\label{ip-psiphi-li} The function $\phi$ is continuous and the function $\psi$ is measurable, moreover for every fixed $t\in[0,T]$ the map $\psi(t,\cdot,\cdot,\cdot):H\times\R\times U\rightarrow\R$
is continuous. There exist nonnegative constants $L_\psi,\,K_\psi,\,m,\,K_\phi$
such that
\begin{align*}
&\vert \psi(t,x,y_1,z_1)- \psi(t,x,y_2,z_2)\vert\leq L_\psi\left(\vert y_1-y_2\vert+\vert z_1-z_2\vert
\right),\\
& \vert \psi(t,x,0,0)\vert\leq K_\psi(1+|x|^m), \qquad 
\vert \phi(x)\vert\leq K_\phi(1+\vert x\vert^m),
\end{align*}
for every $t\in[0,T]$, $x\in H$, $y_1,y_2\in\R$ and $z_1,z_2\in U$.
\end{hypothesis}

\begin{theorem}\label{teo-ex-bsde}Let $(X,Y,Z)$ be solution of the forward-backward system (\ref{fbsde})
and assume that Hypotheses \ref{hyp:base} and \ref{ip-psiphi-li} hold true.
Then there exists a unique solution of the Markovian FBSDE in (\ref{fbsde}) such that, for every $p\in[2,\infty)$, 
\[
 \Vert Y\Vert _{\cals^p}+\Vert Z\Vert_{\calm^p}:=\E[\sup_{s\in[0,T]}|Y_s|^p]+\E\left(\int_0^T|Z_s|^2ds\right)^{\frac p2}\leq C(1+|x|^m),
\]
where $C$ is a positive constant that may depend on $A,\,B,\,G,\, K_\psi,\,L_\psi,\,K_\phi,\,T$ and $p$.
\end{theorem}
%\dim This result substantially follows from \cite{Kob}.

\qed

%Many other results concerning not bounded final data and generator with polynomial growth with respect to $x$have been proved after \cite{Kob}, we cite here \cite{BriHu2006} and \cite{BriHu2008}.

%We also cite the Feynman-Kac formula, proved in \cite{BriFu} when all the coefficients are differentiable and in the case of $\psi$ quadratic with respect to $z$, and generalized e.g. in \cite{Mas1} to nonsmooth coefficients, and in \cite{MR} to the case of $\psi$ superquadratic with respect to $z$.More precisely, l

We now derive a non linear version of the Bismut-Elworthy
formula as the one proved in \cite{futeBismut}, under assumption \eqref{link_gat_der_mall_der}, that in particular holds true if we consider $G$
bounded and invertible.

We also consider the following additional hypothesis.
\begin{hypothesis}
\label{hyp:operatorG}
We assume that $G(t)=J\widetilde G(t)$, where $J\in\mathcal L(U,H)$ is the operator introduced in Hypothesis \ref{hyp:base}(i) and $\widetilde G:[0,T]\to \mathcal L(U)
$ is such that $\widetilde G(t)$ is invertible with bounded inverse for every $t\geq 0$, i.e., and $\sup_{t\in[0,T]}\|\widetilde G(t)^{-1}\|_{\mathcal L(U)}<\infty$. Moreover, there exists a positive constant $C$ such that
\begin{align}
\label{stima_norma_L2_utilde}
\|\widetilde u_{Ja,s,x,r}\|_{L^2(s,r;U)}\leq \frac{C\|a\|_U}{(r-s)^{\frac12}}, \qquad 0\leq s<r<T.
\end{align}    
\end{hypothesis}

To enlighten the notation, we set
\begin{equation}\label{girsanov-density}
\bm{\Psi}:=\exp\left\{
{\displaystyle-\int_{s}^{T}}
\left\langle \widetilde G(r)^{-1}\bar B\left( r, X_r^{s,x}\right)  ,dW_{r}\right\rangle_U -\frac{1}{2}
{\displaystyle\int_{s}^{T}}
\left\|  \widetilde G(r)^{-1}\bar B\left( r, X_r^{s,x}\right)    \right\|  _{U}^{2}ds\right\},
\end{equation}
where $\bar B$ is the function introduced in Hypothesis \ref{hyp:base}(i), and we denote by $\widetilde \E $ the expectation with respect to the probability measure $\widetilde \P $ defined by $\dfrac{d\widetilde \P}{d \P}:=\bm{\Psi}$, and by $\widetilde W$ the $U$-cylindrical Wiener process defined as $d\widetilde W_t=dW_t+\widetilde G(t)^{-1}\bar B(t,X^{s,x}_t)dt$. 

Finally, we recall the following approximation result which will be exploited in the proof of Theorem \ref{teoBismutLip}. A proof of this result can be found in \cite[Lemma 4.3]{futeBismut}.
\begin{lemma}
\label{lem:approx_psi_phi}
Let $\psi$ and $\phi$ satisfy Hypothesis \ref{ip-psiphi-li}. Hence, there exist sequences $(\psi_n)$ and $(\phi_n)$ such that
\begin{enumerate}[\rm(i)]
\item for every $n\in\N$ we have $\phi_n\in\mathcal G^{1}(H,\R)$ and $\psi(t,\cdot,\cdot,\cdot)\in\mathcal G^{1,1,1}(H\times \R\times U,\R)$ for every $t\in[0,T]$;
\item for every $n\in\N$ there exist constants $C_n>0$ and $m_n\geq0$ such that 
\begin{align*}
|\nabla_x\psi_n(t,x,y,z)h|\leq C_n|h|(1+|z|)(1+|x|+|y|)^{m_n}, \qquad |\nabla_x\phi_n(x)h|\leq C_n|h|
\end{align*}
for every $t\in[0,T]$ and every $(x,y,z,h)\in H\times \R\times U\times H$;
\item there exist positive constants $\widetilde L_\psi$, $\widetilde K_\psi$ and $ \widetilde K_\phi$, which only depend on $L_\psi,K_\psi$, $K_\phi$ and $m$, such that
\begin{align*}
& |\psi_n(t,x,y_1,z_1)-\psi_n(t,x,y_2,z_2)|\leq \widetilde L_\psi(|y_1-y_2|+|z_!-z_2|), \\
& |\psi_n(t,x,0,0)\leq K_\psi(1+|x|^m), \qquad |\phi_n(x)|\leq K_\phi(1+|x|^m)
\end{align*}
for every $n\in\N$, every $t>0$ and every $(x,y,z)\in H\times \R\times U$;
\item $(\psi_n)$ and $(\phi_n)$ pointwise converge to $\psi$ and $\phi$, respectively, as $n$ tends to infinity.
\end{enumerate}
\end{lemma}
\begin{theorem}\label{teoBismutLip}
Assume that Hypotheses \ref{hyp:base}, \ref{ip-psiphi-li} and \ref{hyp:operatorG} and condition \eqref{link_gat_der_mall_der} hold true.
Then for every $0\leq s< t\leq T$ and $x,h\in H$ 
% such that $\alpha$
% \begin{align*}
% [t,T]\ni r\mapsto \|\widetilde u_{h,s,x,r}\|_{L^2([s,r];U)}\in L^1([t,T]),    
% \end{align*}
we get the following Bismut formula
\begin{align}\label{Bismut-lip}
\widetilde{\E}\left[ \langle \nabla_x\,Y^{s,x}_t,h\rangle \right]
= & 
\widetilde{\E}\int_t^T\psi\left(r,X_r^{s,x},Y_r^{s,x},Z_r^{s,x}\right)
\int_s^r\langle \widetilde u_{h,s,x,r}(\sigma),d\widetilde W_\sigma\rangle_{U}\,dr \notag \\
&+\widetilde{\E}\displaystyle\int_t^T Z_r\widetilde G(r)^{-1}\bar B(r,X_r^{s,x})\int_s^r\langle\widetilde u_{h,s,x,r}(\sigma),d\widetilde W_\sigma\rangle_U dr \notag \\
& 
+\widetilde{\E}\left[  \phi(X_T^{s,x})\int_s^T\langle \widetilde u_{h,s,x,T}(\sigma),d\widetilde W_\sigma\rangle_{U}\right],
\end{align}
where $\widetilde \E $ is the expectation with respect to the probability measure $\widetilde \P$ such that $\dfrac{d\widetilde \P}{d \P}:=\bm{\Psi}$, where $\bm\Psi$ is defined in \eqref{girsanov-density}.
\newline Moreover, if $h=Ja$, $a\in U$, then equality \eqref{Bismut-lip} holds true also for $t=s$ and there exists a positive constant $C$, which depends on $A,\bar B,\widetilde G,J,K_\phi,K_\psi,L_\psi$ and $T$, such that
\begin{align}
\label{stima_Gder_Y_lip}
|\langle \nabla_xY_s^{s,x},Ja\rangle_U|
\leq \frac{C(1+|x|^m)|a|_U}{(T-s)^\frac12}, \qquad s\in[0,T), \ a\in U.
\end{align}
\end{theorem}
\begin{proof}
At first, we assume that $\phi,\psi$ and $\bar B$ are differentiable with respect to $x,y$ and $z$. Finally, we will remove such an assumption by means of an approximating procedure. 

We begin by considering the case $\bar B=0$.
For all $h\in H$, the process $\nabla_xX^{s,x}h=(\nabla_x X_t^{s,x}h)_{t\in[0,T]}$ satisfies, $\mathbb P$-a.s., equation \eqref{mild_solu_first_var}. From \eqref{link_gat_der_mall_der}, we deduce that for every $h$ there exists a function $\widetilde u_{h,s,x,t}\in L^2([0,T];U)$ such that $\nabla_x X^{s,x}_th=\langle DX^{s,x}_t,\widetilde u_{h,s,x,t}\rangle_{L^2(0,T;U)}$ $\mathbb P$-a.s. for every $t\in(s,T]$. 
\newline 
Let us write in integral form the FBSDE in the system \eqref{fbsde}:
\begin{equation}
\label{bsde-int}
Y_t^{s,x}=\displaystyle\int_t^T\psi(r,X_r^{s,x},Y_r^{s,x},Z_r^{s,x})\;dr-\displaystyle\int_t^T Z_r\;dW_r+\phi(X_T^{s,x}), \qquad t\in[s,T].
\end{equation}
Recall also that $\forall t\in [s,T],\,\xi\in H$, the function $(t,\xi)\mapsto Y_t^{t,\xi}$ is deterministic and we set
\begin{equation*} 
%\label{v}
v(t,\xi):=Y_t^{t,\xi}.
\end{equation*}
By the Markov property, we have also that
$Y_t^{s,x}=Y_t^{t,X_t^{s,x}}=v(t,X_t^{s,x})$. 
Moreover, for every $t\in [s,T]$ and $\xi\in H$, since $Z_t^{t,\xi} =\nabla_\xi v(t,\xi)G(t)$ (see \cite[Proposition 3.6(iii)]{futeBismut}), the function $(t,y)\mapsto Z_t^{t,\xi}$ is deterministic. We set
\begin{equation*} 
%\label{vZ}
Z_t^{t,\xi}:=\bar v(t,\xi)
\end{equation*}
and, again by the Markov property, we have also that
$Z_t^{s,x}=Z_t^{t,X_t^{s,x}}=\bar v(t,X_t^{s,x})$.
\newline 
If we differentiate the integral FBSDE \eqref{bsde-int} with respect to $x\in H$, the initial datum of the forward equation in \eqref{fbsde}, consider the regularity of $X,Y$ and $Z$ (see \cite[Proposition 3.5]{futeBismut}) and take the expectation, then we get, for every $h \in H$ and $t\in[s,T]$,
\begin{align}
\label{bsde-int-diff}
\E\langle \nabla_x Y_t^{s,x}, h\rangle_H
= & \E[\nabla\phi(X_T^{s,x})\langle\nabla_x X_T^{s,x},h\rangle_H]+\E\displaystyle\int_t^T\left[\nabla_x\psi(r,X_r^{s,x},Y_r^{s,x},Z_r^{s,x})\langle\nabla_x X_r^{s,x},h\rangle_H\right. \notag
\\
&+\nabla_y\psi(r,X_r^{s,x},Y_r^{s,x},Z_r^{s,x})\nabla_\xi v(r, X_r^{s,x})\langle\nabla_x X_r^{s,x},h\rangle_H  \notag \\ 
& \left.+\nabla_z\psi(r,X_r^{s,x},Y_r^{s,x},Z_r^{s,x})\nabla_\xi \bar v(r, X_r^{s,x})\langle\nabla_x X_r^{s,x},h\rangle_H\right]dr.
\end{align}
We know that there exist functions $\widetilde u_{h,s,x,r}\in L^2([0,T];U)$, $r\in(s,T]$, such that
\begin{equation*}
\langle\nabla_x X_r^{s,x},h\rangle_H=\int_s^rD_\sigma X_r^{s,x}\widetilde u_{h,s,x,r}(\sigma)d\sigma, \qquad r\in(s,T].
\end{equation*}
So, substituting into \eqref{bsde-int-diff} we get
\begin{align*}
\E\langle \nabla_x Y_t^{s,x}, h\rangle_H&=\E\int_s^T\nabla\phi(X_T^{s,x})D_\sigma X_T^{s,x}\widetilde u_{h,s,x,T}(\sigma)\,d\sigma
\\ \nonumber &+\E\displaystyle\int_t^T\left[\nabla_x\psi(r,X_r^{s,x},Y_r^{s,x},Z_r^{s,x})\int_s^rD_\sigma X_r^{s,x}\widetilde u_{h,s,x,r}(\sigma)\,d\sigma\right. 
\\
\nonumber&+\nabla_y\psi(r,X_r^{s,x},Y_r^{s,x},Z_r^{s,x})\nabla_\xi v(r, X_r^{s,x})\int_s^rD_\sigma X_r^{s,x}\widetilde u_{h,s,x,r}(\sigma)\,d\sigma\\&+\left.\nabla_z\psi(r,X_r^{s,x},Y_r^{s,x},Z_r^{s,x})\nabla_\xi \bar v(r, X_r^{s,x})\int_s^rD_\sigma X_r^{s,x}\widetilde u_{h,s,x,r}(\sigma)d\sigma \right]dr
\end{align*}
for every $t\in[s,T]$. Taking into account the Malliavin differentiability of $Y$ and $Z$ (see \cite[Proposition 3.6(ii)]{futeBismut}), we infer that
\begin{equation*}
D_\sigma Y_r^{s,x} =D_\sigma (v(r,X_r^{s,x}))=\nabla_\xi v(r, X_r^{s,x})D_\sigma X_r^{s,x}
\end{equation*}
and
\begin{equation*}
D_\sigma Z_r^{s,x} =D_\sigma (\bar v(r,X_r^{s,x}))=\nabla_\xi \bar v(r, X_r^{s,x})D_\sigma X_r^{s,x}
\end{equation*}
for every $r\in(s,T]$ and $\sigma\in[0,T]$. Integrating by parts and recalling the expression of $\delta(\widetilde u_{h,s,x,r})$, $r\in(s,T]$, we obtain
\begin{align}
\label{bsde-int-derMall}
\E\langle \nabla_x Y_t^{s,x}, h\rangle_H
&=\E\int_s^TD_\sigma\phi(X_T^{s,x})\widetilde u_{h,s,x,T}(\sigma)\,d\sigma
+\E\displaystyle\int_t^T \int_s^rD_\sigma\psi(r,X_r^{s,x},Y_r^{s,x},Z_r^{s,x})\widetilde u_{h,s,x,r}(\sigma)\,d\sigma\,dr \nonumber \\
& =\E\phi(X_T^{s,x})\int_s^T\langle \widetilde u_{h,s,x,T}(\sigma),\,dW_\sigma\rangle_U
+\E\displaystyle\int_t^T \psi(r,X_r^{s,x},Y_r^{s,x},Z_r^{s,x})\int_s^r\langle\widetilde u_{h,s,x,r}(\sigma),dW_\sigma\rangle_U dr
\end{align}
for every $t\in[s,T]$.

Now let us go back to \eqref{fbsde} with $B\neq 0$. From Hypothesis \ref{hyp:base}(i), there exists $\bar B:[0,T]\times H\to U$ such that for every $\tau \in [0,T]$ and $x \in H$, $B(\tau,x)=J\bar B(\tau,x)$, so by a Girsanov change of probability measure in \eqref{fbsde} there exists a probability measure $\widetilde{\mathbb{P}}$ such that the process $\widetilde W$ given by $d\widetilde W_t=dW_t+\widetilde G(t)^{-1}\bar B(t,X_t)dt$ is a $U$-valued cylindrical Wiener process in the probability space $(\Omega,\mathcal F,\widetilde{ \mathbb P})$ and the FBSDE \eqref{fbsde} can be rewritten as
\begin{equation}\label{fbsde-tilde}
\left\{\begin{array}{ll}\dis  dX_t =
AX_tdt+G(t)d\widetilde W_t,&  t\in
[s,T]\subset [0,T],\\[1mm]
X_\tau=x, & \tau\in[0,s], \\[1mm]
-dY_t =\psi(t,X_t,Y_t,Z_t)\;dt+Z_t\widetilde G(t)^{-1}\bar B (t,X_t)\,dt-Z_t\;d\widetilde W_t,
 & t\in[0,T], \\[1mm]
Y_T=\phi(X_T).
\end{array}\right.
\end{equation}
Now we apply the previous results to the FBSDE \eqref{fbsde-tilde} and we get, by \eqref{bsde-int-derMall}, 
\begin{align}
\label{bsde-int-derMall-tilde-E}
\widetilde\E\langle \nabla_x Y_t^{s,x}, h\rangle
=&\widetilde\E\phi(X_T^{s,x})\int_s^T\langle \widetilde u_{h,s,x,T}(\sigma), d\widetilde W_\sigma\rangle_U
+\widetilde\E\displaystyle\int_t^T \psi(r,X_r^{s,x},Y_r^{s,x},Z_r^{s,x})\int_s^r\langle\widetilde u_{h,s,x,r}(\sigma),d\widetilde W_\sigma\rangle_U dr \nonumber \\
&+\widetilde\E\displaystyle\int_t^T Z_r^{s,x}\widetilde G(r)^{-1}\bar B(r,X_r^{s,x})\int_s^r\langle\widetilde u_{h,s,x,r}(\sigma),d\widetilde W_\sigma\rangle_Udr.
\end{align}
We take $t=s$, we recall that $\nabla_xY_{s}^{s,x}$ is deterministic and $Z_r^{s,x}=Z_r^{r,X_r^{s,x}}=\nabla_xY_r^{r,X_r^{s,x}}G(r)$ and consider the norm
\begin{align*}
\||\nabla_xYJ\||:=\sup_{t\in[T-\eta,T)}\sup_{x\in H}(T-t)^\frac12(1+|x|)^{-m}\sup_{a\in U,\  |a|\leq 1}|\langle \nabla_xY(t,t,x),Ja\rangle_H|,
\end{align*}
where $\eta\in(0,T]$ has to be properly chosen. From \eqref{bsde-int-derMall-tilde-E}, for every $h=Ja$, $a\in U$, and $s\in[T-\eta,T)$ we infer that
\begin{align*}
|\langle \nabla_xY_s^{s,x},Ja\rangle_H|
\leq & \sqrt{2}K_\phi\widetilde \E[1+|X_T^{s,x}|^{2m}]^{\frac12}|\|\widetilde u_{Ja,s,x,T}\|_{L^2([s,T],U)}+2K_\psi\int_s^T\|\widetilde u_{Ja,s,x,r}\|_{L^2([s,r];U)}dr \\
& +2K_\psi\int_s^T\widetilde \E[|X_r^{s,x}|^{2m}]^{\frac12}\|\widetilde u_{Ja,s,x,r}\|_{L^2([s,r];U)}dr \\ 
& +2L_\psi\int_s^T\widetilde \E [|Y_r^{s,x}|^2 +|\nabla_xY_r^{r,X_r^{s,x}} G(r)|^2]^{\frac12}\|\widetilde u_{Ja,s,x,r}\|_{L^2([s,r],U)}dr \\
& + \|\bar B\|_\infty\int_s^T\widetilde{\mathbb E}[|\nabla_xY_r^{r,X_r^{s,x}}J|^2]^{\frac12}\|\widetilde u_{Ja,s,x,r}\|_{L^2([s,r];U)}dr \\
\leq & \frac{CK_\phi(1+|x|^m)|a|_U}{(T-s)^\frac12}+C(K_\psi+L_\psi)(1+|x|^m)(T-s)^{\frac12}|a|_U \\
& +
C(\|\bar B\|_\infty+L_\psi\sup_{t\in[0,T]}\|\widetilde G(t)\|_{\mathcal L(U)})\||\nabla_xYJ\||\widetilde {\mathbb E}\int_s^T\frac{(1+|X_r^{s,x}|_H^m)}{(T-r)^{\frac12}(r-s)^\frac12} ds \\
\leq & \frac{C(1+|x|^m)K_\phi|a|_U}{(T-s)^\frac12}+C(K_\psi+L_\psi)(1+|x|^m)(T-s)^{\frac12}|a|_U \\
& +C(\|\bar B\|_\infty+L_\psi\sup_{t\in[0,T]}\|\widetilde G(t)\|_{\mathcal L(U)})\||\nabla_xYJ\||(1+|x|_H^m)|a|_U,
\end{align*}
for some positive constant $C$, where we have used Hypothesis \ref{ip-psiphi-li} and estimate \eqref{stima_norma_L2_utilde}. It follows that
\begin{align*}
\||\nabla_xYJ\||
\leq C(K_\phi+(K_\psi+\widetilde L_\psi)\eta)+C(\|\bar B\|_\infty+\widetilde L_\psi\sup_{t\in[0,T]}\|\widetilde G(t)\|_{\mathcal L(U)})\||\nabla_xYJ\||\eta^{\frac12}.
\end{align*}
Then, it is possible to choose $\eta\in(0,T]$ such that
\begin{align*}
\||\nabla_xYJ\||
\leq 2C(K_\phi+(K_\psi+\widetilde L_\psi)\eta).
\end{align*}
Iterating this argument, it follows that there exists a positive constant $C$, which only depends on $A,\bar B,\widetilde G,G,K_\phi,K_\psi,L_\psi$ and $T$, such that
\begin{align*}
|\langle \nabla_xY_s^{s,x},Ja\rangle_U|
\leq \frac{C(1+|x|^m)|a|_U}{(T-s)^\frac12}, \qquad s\in[0,T), \ a\in U.
\end{align*}
To prove the result in the general case, i.e., when $B$ satisfies Hypothesis \ref{hyp:base}(i) and $\phi$ and $\psi$ satisfy Hypothesis \ref{ip-psiphi-li}, can be obtained by considering the solution $(X,Y_n,Z_n)$ to the FBSDE \eqref{fbsde}, with $\phi$, $\psi$ and $\bar B$ replaced by $\phi_n$, $\psi_n$ and $\bar B_n$, $n\in\N$, respectively, and arguing as in the final part of the proof of \cite[Theorem 4.2]{futeBismut}. We simply notice that the approximating sequence $(\bar B_n)$ can be defined by means of convolution (see \cite[Lemma 4.3]{futeBismut}), that its supremum norm is $\|\bar B\|_\infty$ and that, in the proof of the convergence, the difference $Z_{n,r}^{s,x}\widetilde G(r)^{-1}\bar B_n(r,X_r^{s,x})-Z_{r}^{s,x}\widetilde G(r)^{-1}\bar B(r,X_r^{s,x})$ can be estimated by
\begin{align*}
|Z_{n,r}^{s,x}&\widetilde G(r)^{-1} \bar B_n(r,X_r^{s,x})-Z_{n,r}^{s,x}\widetilde G(r)^{-1}\bar B_n(r,X_r^{s,x})|\\
\leq & |Z_{n,r}^{s,x}\widetilde G(r)^{-1}\bar B_n(r,X_r^{s,x})-Z_{r}^{s,x}\widetilde G(r)^{-1}\bar B_n(r,X_r^{s,x})| \\
& \!+\!|Z_{r}^{s,x}\widetilde G(r)^{-1}\bar B_n(r,X_r^{s,x})-Z_{r}^{s,x}\widetilde G(r)^{-1}\bar B_n(r,X_r^{s,x})| \\
\leq & \|\bar B\|_\infty|Z_{n,r}^{s,x}-Z_{r}^{s,x}|+|Z_{r}^{s,x}B_n(r,X_r^{s,x})-Z_{r}^{s,x}B(r,X_r^{s,x})|
\end{align*}
and the result follows by means of the convergence of $(Y_{n}^{s,x},Z_{n}^{s,x})$ to $(Y^{s,x},Z^{s,x})$ in $\mathcal S^2\times \mathcal M^2$, the pointwise convergence of $(B_n)$ to $B$ and the dominated convergence theorem.
\end{proof}

\begin{remark}\label{rm:Ga} Notice that we need that the apriori estimate on the derivative of $Y$ is integrable in $(0,T)$. Since in our examples, the full derivative of $Y$ blows up as $t^{-\frac32}$ (or worst) near $0$, this is the reason why we confine ourself to only consider the derivative of $Y$ along the direction of $J$.
\end{remark}

Under a different set of assumptions, we deduce the differentiability of $Y$ along other directions.
\begin{hypothesis}
\label{hyp:operatorG1}
We assume that $G(t)=J\widetilde G(t)$, where $J\in\mathcal L(U,H)$ is the operator introduced in Hypothesis \ref{hyp:base}(i) and $\widetilde G:[0,T]\to \mathcal L(U)
$ is such that $\widetilde G(t)$ is invertible with bounded inverse for every $t\geq 0$, i.e., $\sup_{t\in[0,T]}\|\widetilde G(t)^{-1}\|_{\mathcal L(U)}<\infty$. We further assume that $J=J_1S$, $J$ being the operator introduced in Hypothesis \ref{hyp:base}$(i)$, where $J_1\in\mathcal L(U,H)$, $S\in\mathcal L(U)$, and there exists positive constants $C,\beta$, with $\beta<\frac12$, such that
\begin{align}
\label{stima_norma_L2_utilde_1}
\|\widetilde u_{J_1a,s,x,r}\|_{L^2(s,r;U)}\leq \frac{C\|a\|_U}{(r-s)^{\frac12+\beta}}, \qquad 0\leq s<r<T.
\end{align}    
\end{hypothesis}

We set 
\begin{equation*}
\bm{\Psi}:=\exp\left\{
-{\displaystyle\int_{s}^{T}}
\left\langle \widetilde G(r)^{-1}P\bar B\left( r, X_r^{s,x}\right)  ,dW_{r}\right\rangle_U -\frac{1}{2}
{\displaystyle\int_{s}^{T}}
\left\|  \widetilde G(r)^{-1}P\bar B\left( r, X_r^{s,x}\right)    \right\|  _{U}^{2}ds\right\},
\end{equation*}
where $P$ is the projection on ${\rm Ker}(S)^\perp$, $\bar B$ is the function introduced in Hypothesis \ref{hyp:base}(i) and we denote by $\widetilde \E $ the expectation with respect to the probability measure $\widetilde \P$ defined by $\dfrac{d\widetilde \P}{d\P}:=\bm{\Psi} $, and by $\widetilde W$ the $U$-cylindrical Wiener process defined as $d\widetilde W_t=dW_t+\widetilde G(t)^{-1}P\bar B(t,X^{s,x}_t)dt$.

\begin{theorem}\label{teoBismutLip1}
Assume that Hypotheses \ref{hyp:base}, \ref{ip-psiphi-li} and \ref{hyp:operatorG1} and condition \eqref{link_gat_der_mall_der} hold true.
Then for every $0\leq s< t\leq T$ and $x,h\in H$ 
we get the following Bismut formula
\begin{align}\label{Bismut-lip2}
\widetilde{\E}\left[ \langle \nabla_x\,Y^{s,x}_t,h\rangle \right]
= & 
\widetilde{\E}\int_t^T\psi\left(r,X_r^{s,x},Y_r^{s,x},Z_r^{s,x}\right)
\int_s^r\langle \widetilde u_{h,s,x,r}(\sigma),d\widetilde W_\sigma\rangle_{U}\,dr \notag \\
&+\widetilde{\E}\displaystyle\int_t^T Z_r\widetilde G(r)^{-1}P\bar B(r,X_r^{s,x})\int_s^r\langle\widetilde u_{h,s,x,r}(\sigma),d\widetilde W_\sigma\rangle_U dr \notag \\
& 
+\widetilde{\E}\left[  \phi(X_T^{s,x})\int_s^T\langle \widetilde u_{h,s,x,T}(\sigma),d\widetilde W_\sigma\rangle_{U}\right].
\end{align}
Moreover, if $h=Ja$, $a\in U$, then equality \eqref{Bismut-lip} holds true also for $t=s$ and there exists a positive constant $C$, which depends on $A,\bar B,\widetilde G,J_1,S,K_\phi,K_\psi,L_\psi$ and $T$, such that
\begin{align}
\label{stima_Gder_Y_lip_1}
|\langle \nabla_xY_s^{s,x},Ja\rangle_U|
\leq \frac{C(1+|x|^m)|a|_U}{(T-s)^{\frac12+\beta}}, \qquad s\in[0,T), \ a\in U.
\end{align}
\end{theorem}
\begin{proof}
The proof follows the line of that of Theorem \ref{teoBismutLip}, noticing that $G(t)=J_1S\widetilde G(t)$ and $B(t,x)=J_1S\bar B(t,x)$ for every $t\in[0,T]$ and $x\in H$. Hence, we can apply Girsanov theorem to obtain system \eqref{fbsde-tilde} with $d\widetilde W_t=dW_t+\widetilde G(t)^{-1}P\bar B(t,X_t)dt$.

Then, one considers the norm
\begin{align*}
\||\nabla_xYJ_1\||:=\sup_{t\in[T-\eta,T)}\sup_{x\in H}(T-t)^{\frac12+\beta}(1+|x|)^{-m}\sup_{a\in U,\  |a|\leq 1}|\langle \nabla_xY(t,t,x),J_1a\rangle_H|,
\end{align*}
where $\eta\in(0,T]$ has to be properly chosen, obtaining
\begin{align*}
|\langle \nabla_xY_s^{s,x},J_1a\rangle_H|
\leq & \sqrt{2}K_\phi\widetilde \E[1+|X_T^{s,x}|^{2m}]^{\frac12}|\|\widetilde u_{J_1a,s,x,T}\|_{L^2([s,T],U)}+2K_\psi\int_s^T\|\widetilde u_{J_1a,s,x,r}\|_{L^2([s,r];U)}dr \\
& +2K_\psi\int_s^T\widetilde \E[|X_r^{s,x}|^{2m}]^{\frac12}\|\widetilde u_{J_1a,s,x,r}\|_{L^2([s,r];U)}dr \\ 
& +2L_\psi\int_s^T\widetilde \E [|Y_r^{s,x}|^2 +|\nabla_xY_r^{r,X_r^{s,x}} G(r)|^2]^{\frac12}\|\widetilde u_{J_1a,s,x,r}\|_{L^2([s,r],U)}dr \\
& + \|S\bar B\|_\infty\int_s^T\widetilde{\mathbb E}[|\nabla_xY_r^{r,X_r^{s,x}}J_1|^2]^{\frac12}\|\widetilde u_{J_1a,s,x,r}\|_{L^2([s,r];U)}dr \\
\leq & \frac{C(1+|x|^m)K_\phi|a|_U}{(T-s)^{\frac12+\beta}}+C(K_\psi+L_\psi)(1+|x|^m)(T-s)^{\frac12-\beta}|a|_U \\
& +C(\|S\bar B\|_\infty+L_\psi\sup_{t\in[0,T]}\|S\widetilde G(t)\|_{\mathcal L(U)})\||\nabla_xYJ_1\||(1+|x|_H^m)|a|_U(T-s)^{-2\beta}.
\end{align*}
This gives
\begin{align*}
\||\nabla_xYJ_1\||
\leq C(K_\phi+(K_\psi+\widetilde L_\psi)\eta)+C(\|S\bar B\|_\infty+\widetilde L_\psi\sup_{t\in[0,T]}\|S\widetilde G(t)\|_{\mathcal L(U)})\||\nabla_xYJ_1\||\eta^{\frac12-\beta}.
\end{align*}
The conclusion of the proof can be derived arguing as above.
\end{proof}

\begin{remark}
Notice that, having proved a Bismut-Elworthy type formula for the FBSDE \eqref{fbsde} as a byproduct we get a Bismut-Elworthy type formula for the transition semigroup of a perturbed Ornstein Uhlenbeck process, without  asking differentiability of the drift. Indeed we can proceed as in the proof of Theorem \ref{teoBismutLip}: we start from the forward backward system \eqref{fbsde} with $\psi=0$ and we perform the Girsanov change of measure to arrive at
\begin{equation*}
\left\{\begin{array}{ll}\dis  dX_t =
AX_tdt+G(t)d\widetilde W_t,&  t\in
[s,T]\subset [0,T],\\[1mm]
X_\tau=x, & \tau\in[0,s], \\[1mm]
-dY_t =Z_t\bar B(t,X_t)dt-Z_t\;d\widetilde W_t,
 & t\in[0,T], \\[1mm]
  Y_T=\phi(X_T),
\end{array}\right.
\end{equation*}
getting the Bismut formula \eqref{Bismut-lip} with $\psi=0$:
\begin{align*}
\widetilde{\E}\left[ \langle \nabla_x\,Y^{s,x}_t,h\rangle \right]
= 
&+\widetilde{\E}\displaystyle\int_t^T Z_r\widetilde G(r)^{-1}\bar B(r,X_r^{s,x})\int_s^r\langle\widetilde u_{h,s,x,r}(\sigma),d\widetilde W_\sigma\rangle_U dr \notag \\
& 
+\widetilde{\E}\left[  \phi(X_T^{s,x})\int_s^T\langle \widetilde u_{h,s,x,T}(\sigma),d\widetilde W_\sigma\rangle_{U}\right].
\end{align*}
Notice that in the case of $G$ invertible with bounded inverse the result is true also in the case of multiplicative noise, that is $G$ can depend on also on $x$ 
and this is an improvement towards \cite{BoFu} and \cite{Ma07} since we don't require the drift to be differentiable.
\newline 
We underline that if equation \eqref{eq_nonlin_1} is well posed under regularity assumptions on the drift weaker than Lipschitz continuity, then Bismut formula \eqref{Bismut-lip} proved in Theorem \ref{teoBismutLip} remains true. This is the case of the wave equation, which is well posed if the drift is $\alpha$-H\"older continuous with $\alpha\in\left(\frac23,1\right)$, see \cite{MaPr,MaPr24}, and of the damped wave equation, which is well posed if the drift is $\theta$-H\"older continuous for a suitable $\theta\in (0,1)$, see \cite{AddBig24,AddBig26,AddMasPri23}.
\end{remark}

\section{Connections with PDEs}
\label{sec:pde}

In this section we aim to apply the Bismut formula obtained in Theorem \ref{teoBismutLip} to the solution of the semilinear Kolmogorov equation.
\newline We start by proving that $Z$ is the directional derivative of $Y$, in a sense that we make precise in the following proposition. %\textcolor{red}{
Moreover, we recall that, due to the invertibility assumption on $\widetilde G(s)$ (see Hypothesis \ref{hyp:operatorG}), derivations along $G(s)$ are equivalent with derivations along $G$.

\begin{proposition}\label{prop-identif-Z}
Under the assumptions of Theorem \ref{teoBismutLip}, $\forall\, s\in[0,T)$ and almost every $x\in H$,
\begin{equation}\label{identif-Z}
Z^{s,x}_s = \nabla_x\,Y^{s,x}_sG(s).
\end{equation}
    
\end{proposition}
\dim We start by recalling that  $\forall s\in [0,T],\,x\in H$, the function $(s,x)\mapsto Y_s^{s,x}$ is deterministic and we set
\begin{equation*} 
%\label{v}
v(s,x):=Y_s^{x,x}, \qquad x\in H, \ s\in[0,T].
\end{equation*}
By the Markov property, we have also that
$Y_t^{s,x}=Y_t^{t,X_t^{s,x}}=v(t,X_t^{s,x})$ for every $0\leq s<t\leq T$ and $x\in H$. 
%%%%%%%%%%%%%
%Moreover, see \cite{FuhTes2002} and also \cite[Proposition 3.6(iii)]{futeBismut}, if the coefficients are differentiable  for every $s\in [0,T]$ and $x\in H$, $Z_s^{s,x} =\nabla_x v(s,x)G(s)$ (see \cite[Proposition 3.6(iii)]{futeBismut}), the function $(t,y)\mapsto Z_t^{t,\xi}$ is deterministic. We set
%\begin{equation} \label{vZ}Z_t^{t,\xi}:=\bar v(t,\xi)
%end{equation}
%and, again by the Markov property, we have also that$Z_t^{s,x}=Z_t^{t,X_t^{s,x}}=\bar v(t,X_t^{s,x})$
%%%%%%%%%%%%%%%%%%%%%%%%%%%%
It is well known (see \cite{FuhTes2002}) that if the coefficients $\psi$, $\bar B$ and $\phi$ are differentiable, then $Z_s^{s,x} =\nabla_x Y^{s,x}_s G(s)$.
We have just shown in the proof of Theorem \ref{teoBismutLip} that, with coefficients $\psi$ and $\phi$ only continuous, under the assumptions of Theorem \ref{teoBismutLip}, $x\rightarrow Y_s^{s,x}$ is differentiable in the directions $Ga,\,a\in U$. 
\newline Now let $\phi$ and $\psi$ be approximated by their inf-sup convolutions $\phi_n$ and $\psi_n$, and let  $(Y^{n,s,x},Z^{n,s,x})$ be the solution to the FBSDE in (\ref{fbsde}) with generator $\psi_n$ in the place of $\psi$ and final datum $\phi_n$ in the place of $\psi$; again by theorem \ref{teoBismutLip} for every $a \in U$ and every $s\in[0,T)$ the sequence $(\nabla_xY^{n,s,x}_s G(s)a)_{n\in\N}$ converges to $\nabla_xY^{s,x}_s Ga$ almost surely in $\Omega$, which implies that 
$$
\nabla_xY^{n,s,x}_s G(s)a\rightarrow \nabla_xY^{s,x}_s G(s)a.
$$
%By the aforementioned identification of $Z$ with the directional derivative of $Y$ we have that $Z_s^{n,s,x} =\nabla_x Y^{n,s,x}_s G(s),\; \text{ for a.a. } s\in[0,T]$.
Moreover, we know that, by the aforementioned identification of $Z^n$ (see \cite{FuhTes2002}), we have that for every $s\in[0,T]$ and $x\in H$, 
\begin{equation}\label{identif-Z-n}
Z^{n,s,x}_s= \nabla_xY^{n,s,x}_s G(s),\quad \P \text{ a.s.}\end{equation}
By the Markov property we know that for any $0\leq s \leq t\leq T$.
\begin{equation}\label{identif-Z-n-1}
Z^{n,s,x}_t=    Z^{n,t,X_t^{s,x}}_t= \nabla Y^{n,t,X_t^{s,x}}_t G(t),
\end{equation}
and so by \eqref{Bismut-lip} it turns out that $Z$ has continuous trajectories. Moreover, for every $0\leq s\leq t\leq T$ it converges to $\nabla Y^{n,t,X_t^{s,x}}_t G(t)$ $\mathbb P$-a.s. Since $Z^{n,s,x} \rightarrow Z^{s,x}$ in $\mathcal M^2$ and then, passing to a subsequence, $Z^{n_j,s,x}_t \rightarrow Z^{s,x}_t,\,\mathbb P-$ a.s. for a.a. $t\in [s,T]$, we deduce that $Z^{s,x}_s=\nabla Y^{s,x}_s G(s)$ $\P -$ a.s. for a.a. $t\in [0,T]$.
\qed

\subsection{Solution of the related Kolmogorov equation}

Now we consider the semilinear nonautonomous Kolmogorov equation associated to the process $X^{s,x}$ solution to \eqref{eq_nonlin_1}. Namely, let $\mathcal L(s) $, at least formally, be given be
$$
(\call(s) f)(x)=\frac{1}{2}(Tr G(s)G^*(s) \nabla^2 f)(x)+\<Ax,\nabla f(x)\>+\<B(s,x),\nabla f(x)\>.
$$
and the associated Cauchy problem
\begin{equation}
\left\{
\begin{array}{ll}
\displaystyle\frac{\partial v}{\partial s}(s,x)=-\call(s) v\left( s,x\right)
+\psi\left( s,x,v(s,x),\nabla v(s,x)G(s)   \right),& s\in\left[  0,T\right]
,\text{ }x\in H, \\[3mm]
v(T,x)=\phi\left(  x\right)  , & x\in H.
\end{array}
\right.  
\label{Kolmo_1}
\end{equation}
We introduce the notion of mild solution of the non linear Kolmogorov
equation (\ref{Kolmo_1}), see e.g. \cite{DP3}. 
%Let $P_{s,t},\,s\leq t\leq T$,the transition semigroup related to the process $X^{t,x}$ solution of the forward equation (\ref{forward}),namely, for every bounded and measurable function $\phi:H\rightarrow\R$
%\[
 %P_{t,\tau}[\phi](x)=\E\phi(X_\tau^{t,x}).\]
Since to $(\call(s))_{s\in[0,T]}$ it is, at least formally, associated the evolution operator 
$(P_{s,t})_{0\leq s\leq t\leq T}$, the variation of constants formula for (\ref{Kolmo_1}) gives:%
\begin{equation}
v(s,x)=P_{s,T}\left[  \phi\right]  \left(  x\right)  +\int_{s}^{T}%
P_{s,t}\left[  \psi(t,\cdot, v(t,\cdot), \nabla v\left( t,\cdot\right)G(t) 
\right]  \left(  x\right)  dt,
\text{\ \ }s\in\left[  0,T\right]  ,\text{ }x\in H. \label{solmildkolmo}%
\end{equation}
We can give the definition of mild solution for the non linear Kolmogorov equation (\ref{Kolmo_1}).
\begin{definition}
\label{defsolmildkolmo}A function $v:\left[  0,T\right]  \times H\rightarrow\mathbb{R}$ is a mild
solution of the non linear Kolmogorov equation (\ref{Kolmo_1}) if $v$ is continuous and $$ (t,x)\mapsto \dfrac{v(t,x)}{(1+\vert x\vert ^2)^{K/2}}$$ belongs to
$\in C_{b}\left(  \left[  0,T\right]  \times H\right)  $, it is differentiable in the directions of $J$ 
%(\textcolor{red}{mettere delle notazioni ancui fare riferimento?})
and equality (\ref{solmildkolmo}) holds.
\end{definition}

\begin{theorem}\label{teo_fey_kac} Assume that Hypotheses \ref{hyp:base} and \ref{ip-psiphi-li} and condition \eqref{link_gat_der_mall_der} hold true, let $(X^{s,x},Y^{s,x},Z^{s,x})$ be the solution of the forward-backward system (\ref{fbsde}) and let
\begin{equation}\label{U}
 \widetilde U_{Ja,s,x,r}:=\int_s^r \widetilde u_{Ja,s,x,r} (t)\,dW_t.
\end{equation}
Then, the nonlinear Kolmogorov equation \eqref{Kolmo_1} has a unique mild solution $v$ given by the formula
$$v(s, x) = Y_s^{s,x} ,\qquad     (s, x) \in [0, T ] \times H.$$
 Moreover, we have, $\P$-a.s.,
\[
Y_s^{t,x} = v(s, X_s^{t,x} ),\quad       Z_s^{t,x} =\nabla_x v(s, X_s^{t,x} )\nabla_x  X_s^{t,x}G(s).
\]
\
and 
\begin{equation}\label{stime-teoKolmo}
\vert v(s,x)\vert\leq C,\qquad \vert \nabla_x Jv(s,x)\vert\leq C\left(T-s\right)^{-1/2},
\end{equation}
with
\begin{equation}\label{Bismut-teoKolmo}
 <\nabla_x\,v(s,x),Ja>_H=
\E\int_s^T\psi\left(r,X_r^{s,x}, v(r,X_r^{s,x}),\nabla_x v(r,X_r^{s,x})G(r)\right) \widetilde U_{Ja,s,x,r}\,dr
+\E\left[  \phi(X_T^{s,x}) \widetilde U_{Ja,s,x,T}\right].
\end{equation}
\end{theorem}
\dim
Let $\phi_n$ and $\psi_n$ be the inf-sup convolution
of $\phi$ and $\psi$ respectively, as introduced in Lemma \ref{lem:approx_psi_phi}, let $(Y^{n,t,x},Z^{n,t,x})$ be the solution
of the BSDE in (\ref{fbsde}) with final datum $\phi_n$ and generator $\psi_n$,  and let $v^n$ be the mild solution of a Kolmogorov equation with final datum $\phi_n$ instead of $\phi$ and semilinear term $\psi_n$ instead of $\psi$, namely
 \begin{equation}
v^n(s,x)=P_{s,T}\left[  \phi_n\right]  \left(  x\right)  +\int_{s}^{T}%
P_{s,r}\left[  \psi_n(r,\cdot, v^n(r,\cdot), \nabla v^n\left(  r,\cdot\right)G(r)) \right]  \left(  x\right)  dr.
\text{\ \ }s\in\left[  0,T\right]  ,\text{ }x\in H. \label{kolmo-mild-n}%
\end{equation}
Since $\phi_n$ and $\psi_n$ are differentiable, by Theorem \ref{teo_fey_kac} we already know that
$$
v^n(s,x)=Y^{n,s,x}_s, \qquad\nabla_x v^n(s,x) G(s)=Z^{n,s,x}_t,
$$
moreover
\[
 \lim_{n\rightarrow\infty} Y^{n,s,x}_s=Y^{s,x}_s,
\]
where $(Y^{t,x},Z^{t,x})$ is a solution to the backward equation in (\ref{fbsde}). By Theorem
\ref{teoBismutLip} we know that $Y$ satisfies (\ref{Bismut-teoKolmo})
and by \ref{prop-identif-Z} we get the identification \eqref{identif-Z} of $Z$ with $\nabla v G$. So it turns out that $v(s, x) = Y_s^{s,x}$  for $(s, x) \in [0, T ] \times H.$
 Moreover, we have, $\P$-a.s.,
\[
Y_s^{t,x} = v(s, X_s^{t,x} ),\quad       Z_s^{t,x} =\nabla_x v(s, X_s^{t,x} )\nabla_x  X_s^{t,x}G(s).
\]
and passing to the limit in \eqref{kolmo-mild-n} we get 
\begin{equation}
v(s,x)=P_{s,T}\left[  \phi\right]  \left(  x\right)  +\int_{t}^{T}%
P_{s,r}\left[  \psi(r,\cdot, v^n(r,\cdot), \nabla v^n\left(  r,\cdot\right)G(r)) \right]  \left(  x\right)  dr.
\text{\ \ }s\in\left[  0,T\right]  ,\text{ }x\in H. \label{kolmo-mild}%
\end{equation}
can be obtained as in proposition \ref{prop-identif-Z}, and the proof is concluded. Uniqueness of the solution follows by proposition \ref{prop-identif-Z} and the link between semilinear Kolmogorov equations and BSDEs, see also \cite{futeBismut}, proof of Theorem 4.2 for more details. \qed

\section{The stochastic wave equation}
\label{sec:exa_wave}
In this section we consider a semilinear stochastic wave equation of the following form
%non-linear stochastic wave equation:
\begin{equation} 
\left\{
\begin{array}
[c]{l}%
\frac{\partial^{2} y}{\partial t^{2}}\left( t \right)  =-\Lambda y\left( t\right)  +C\left(  t,y(t)\right) +\sigma(t)\dot{W}\left( t\right), \;\;\; 
\\[2mm]
%y\left(  \tau,0\right)  =y\left(  \tau,1\right)  =0,\\
y\left(  0\right)  =x_{0} \in U, \;\;\; \;\;
\frac{\partial y}{\partial t}\left(  0\right)  =x_{1} \in V',\;\;\; \tau \in (0,T].
\end{array}
\right.  \label{wave}
\end{equation}
where
$\Lambda : \mathcal D (\Lambda) \subset U \to U $ is  a positive self-adjoint operator on a real separable Hilbert space $U$ such that there exists $\Lambda^{-1}: U \to U$ of trace class  and 
$\left\{ W(t)= W_{t},t\geq0\right\}  $ is a cylindrical Wiener process  on $U$, the drift $C:[0,T] \times U \to U$  is a bounded measurable function, Lipschitz continuous with respect to the second variable, uniformly in $t \in [0,T]$.

An example of equation \eqref{wave} is given by the stochastic wave equation, 
\begin{equation}
 \left\{
 \begin{array}
 [c]{l}%
 \frac{\partial^{2}}{\partial t ^{2}}y\left(  t,\xi\right)  =\frac
 {\partial^{2}}{\partial\xi^{2}}y\left(  t,\xi\right)  + \bar C\left(  t
 ,\xi, y(t, \xi)
% \frac{\partial y\left(  \tau,\xi\right)}{\partial \tau }
\right) +\sigma(t)\dot{W}\left(  t,\xi\right), \;\;\; \xi \in (0,1),\\
 y\left( t,0\right)  =y\left( t,1\right)  =0,\\
 y\left(  0,\xi\right)  =x_{0}\left(  \xi\right)  , \;\;\; 
 \frac{\partial y}{\partial t}\left(  0,\xi\right)  =x_{1}\left(  \xi\right),\;\;\; t \in (0,T],\;\;\; \xi \in [0,1].
 \end{array}
 \right.  \label{wave-concreta}
 \end{equation}
 Here    $b : \left[  0,T\right]  \times\left[  0,1\right]
\times\mathbb{R} \to \R $  is measurable and,  for  $t
\in\left[  0,T\right]  ,$ a.e. $\xi\in\left[  0,1\right]  ,$ the map $b\left(
t,\xi,\cdot\right)  :\mathbb{\mathbb{R}}  \rightarrow\mathbb{R}$ is continuous. Moreover, 
there exists $c_{1} \in L^{\infty}([0,1])$,
 such that, for  $t\in\left[0,T\right]  $ and a.e. $\xi\in\left[  0,1\right],$  
\[
\left| \bar C\left(  t,\xi,x\right) -\bar C\left( t,\xi,y\right)  \right|
\leq c_{1}\left(  \xi\right)  \left|  x-y\right|
%+\vert z-w\vert^\beta
\]
$ x,\,y \in \R$. Further, we require that 
$\left| \bar C\left(  t,\xi,x\right)  \right|  \leq c_{2}\left(  \xi\right)
,$  for $t \in [0,T]$, $x \in \R$ and a.e. $\xi \in [0,1]$, with $c_{2} \in L^2 (\left[  0,1\right] ) $. A similar model, with the drift $\bar C$ H\"older continuous, is discussed in \cite{MaPr24}. 
\newline We assume that the diffusion $\sigma $ in \eqref{wave} and \eqref{wave-concreta} is bounded and invertible.
\begin{hypothesis} \label{ip-sigma-wave}
We assume that $[0,T]\times [0,1]\ni (t,\xi)\mapsto\sigma(t,\xi)$ is continuous, bounded from above and from below by a positive constant.
To lighten the notation, we set 
$\|\sigma\|_\infty := \|\sigma\|_{L^\infty([0,T]\times [0,1])}$.
\end{hypothesis}
 
We set $V=D(\Lambda^{1/2})$ and $V':=\overline{U}^{|\cdot|_{V'}}$ where $|\cdot |_{V'}:=|\Lambda^{-1/2}\cdot|_U$. $V'$ turns out to be a separable Hilbert space, and we still denote by $\Lambda$ the extension of $\Lambda$ on $V'$. In particular, in $V'$ we have $D(\Lambda^{1/2})=U$ and $D(\Lambda)=V$. Further, assume that the embedding $U\subseteq V'$ is Hilbert-Schmidt.

We set $H:=U\times V'$ endowed with inner product
\begin{align*}
\langle\begin{pmatrix}
u_1 \\ v_1
\end{pmatrix}, 
\begin{pmatrix}
u_2 \\ v_2
\end{pmatrix}
\rangle_H:=\langle u_1,u_2\rangle_U+\langle v_1,v_2\rangle_{V'}
\end{align*}
for every $u_1,u_2\in U$ and $v_1,v_2\in V'$. 

%To clarify this we introduce 
Equation \eqref{wave} can be reformulated as an abstract evolution equation in the product space $H:=  U \times  V' $
where the state variable is given by the pair position-velocity, while it is not well posed 
%We recall why we consider the space $H$ instead of
in the more usual space $K = V \times U$. 
%It is well known, see e.g. \cite{DP1}, that the linear stochastic wave equation with additive noise, that is, equation \eqref{wave} with $C\equiv 0$, is well defined in $H$. 
%On the other hand, it is not well-posed in the more usual space $\calh = V \times U$ indeed its solutions do not evolve in $\calh=V \times U$ even if $x_0 \in V$ and $x_1 \in U$. 
The space $H$ is endowed with the inner product $\langle x,y \rangle_H$ 
$= \langle x_1 , y_1  \rangle_U $ + $\langle x_2 , y_2  \rangle_{V'} $ and norm $|x|_H = 
(\langle x,x \rangle_H)^{1/2}$, $x,y \in H$. In the sequel we will also denote $\langle \cdot, \cdot\rangle_H$ and $|\cdot|_H$ by $\langle \cdot, \cdot\rangle $ and $|\cdot|$.
In $H$ one considers the unbounded wave operator $A$  which generates a unitary group 
$e^{tA},\,t\in\R$:
%(cf. \eqref{ci11})
\begin{equation}\label{A}
\mathcal{D}\left(  A\right)  = V 
\times U,
\text{ \ \ \ \ }%
A\left(
\begin{array}
[c]{c}%
y\\
z
\end{array}
\right)  =\left(
\begin{array}
[c]{cc}%
0 & \rm{Id}\\
-\Lambda & 0
\end{array}
\right)  \left(
\begin{array}
[c]{c}%
y\\
z
\end{array}
\right)  ,\text{ \ for every }\left(
\begin{array}
[c]{c}%
y\\
z
\end{array}
\right)  \in\mathcal{D}\left(  A\right),
\end{equation}.
%The operator $A$ is the generator of the contractive group
\[
e^{tA}
\left(
\begin{array}[c]{c}%
y\\z
\end{array}
\right) 
=\begin{pmatrix}
\cos(\Lambda^{1/2}t) & \displaystyle\frac{1}{\Lambda^{1/2}}\sin (\Lambda^{1/2}t) \\ \\
-{\Lambda^{1/2}}\sin (\Lambda^{1/2}t) & \cos(\Lambda^{1/2}t)
\end{pmatrix}\left(
\begin{array}[c]{c}%
y\\z
\end{array}
\right)  ,\text{ \ \ }t\in\mathbb{R},\;\; \left(
\begin{array}
[c]{c}%
y\\
z
\end{array}
\right)  \in H.
\]

We consider the operators $ J\in L(U;H)$ and $\widetilde G(t)$, $t\in[0,T]$, defined as, $\forall\,u\in \,U$,
\begin{align*}
Ju:=\begin{pmatrix}
0 \\  u
\end{pmatrix}, \qquad \widetilde G(t)u=\sigma(t)u.
\end{align*}
By Hypothesis \ref{ip-sigma-wave} we get that $\forall t \in [0,T]$, $\widetilde G(t)$
is invertible with bounded inverse denoted by $\widetilde G(t)^{-1}$.
Equation \eqref{wave} can be written as an abstract evolution equation in $H$ as 
%\textcolor{red}{manca il drift?}
\begin{equation}
\label{wa1}
dX_{t}^{0,x}  = A X_{t}^{0,x} dt+B(t,X_{t}^{0,x}) dt+G(t)dW_t
,\; t\in\left[  0,T\right], \; 
X_{0}^{0,x} =x \in H,
\end{equation} which admits a unique solution in $H$ since thanks to the assumption on $\bar C$ and to \ref{ip-sigma-wave} Hypothesis \ref{hyp:base} holds true with $\gamma_B=\gamma_G=0.$
%it turns out that  $G\in L_2(U,H)$and 
%\begin{equation}\label{sigma-hs}
% \sup_{t \in [0,T]} \| e^{tA} G \|_{L_2(U, H)} < \infty, \;\;\; T>0,
%\end{equation}
%So Hypothesis \ref{hyp:base} (with $\gamma=0$) holds true, so by \cite{FuhTes2002}, see also \cite[Theorem 6.10]{DP1}.
%
%for which we aim at studying well-posedness, our SPDE %(\ref{wave111})
%(\ref{waveequation-holder}) 
%can be rewritten in an abstract way as an
%equation in $H$ of the following form:
%\begin{equation}
%\left\{
%\begin{array}
%[c]{l}%
%dX_\tau  =AX_\tau  d\tau+GB(\tau,X_\tau)d\tau+GdW_\tau  ,%\text{ \ \ \ }\tau\in\left[
%t,T\right] \\
%X_t  =x,
%\end{array}
%\right.  \label{waveeqabstract-holder}%
%\end{equation}
%where 
%\[
% X_\tau :=
% \left(
% \begin{array}{l}

% \textcolor{blue}{\begin{remark}\label{rm:wave-Holder}
% Notice that in \cite{MaPr} and \cite{MaPr24} it is proved that the wave equation \eqref{wave} $($\eqref{wa1} if we consider its abstract formulation$)$ is well posed even if the drift is assumed to be only $\alpha$-H\"older continuous of order $\alpha >2/3$. Moreover the results contained in Theorem \ref{teoBismutLip}, in particular Bismut formula in \eqref{Bismut-lip}, remain true also under H\"older continuity assumption on the drift. Indeed, the way in which we state the Bismut formula only requires existence and uniqueness of the wave equation and boundedness of the drift, but no further regularity on the drift.
% \end{remark}
% }

Now we recall that for the class of stochastic wave equations \eqref{wave}, condition \eqref{link_gat_der_mall_der}, holds true for any $h\in K$, while it is not true
%. Indeed in this case we cannot prove that condition \eqref{link_gat_der_mall_der} holds 
for any $h\in H$, indeed this would imply the strong Feller property for the wave transition semigroup, which is not true, see e. g. \cite{DP2}.

More precisely following \cite{MaPr24}, we have
\begin{proposition}\label{prop:link_gat_der_mall_der}
Let $A$ be defined as in \eqref{A}, let $\Lambda : \mathcal{D} (\Lambda) \subseteq U \to U $ be a positive self-adjoint operator on a separable Hilbert space $U$. Then for any $h\in K $ there exists $\widetilde u_t$ such that 
\begin{align}
\label{link_gat_der_mall_der_onde}
\int_s^te^{(t-r)A} G(r)\widetilde u_t(r)dr= e^{(t-s)A}h, \quad \mathbb P\textup{-a.s.},    
\end{align}
and, for some positive constant $C$ which does not depend on $t\in(0,T]$.
\begin{align*}
\|\widetilde u_t\|_{L^2((0,t);U)}
\leq & C \frac{|h|_{K}}{t^{3/2}};
\end{align*}
moreover if $h=Jb=\begin{pmatrix}
0 \\ b
\end{pmatrix} $  with $b\in U$, then 
\begin{align*}
    \|\widetilde u_t\|_{L^2((0,t);U)}^2
    \leq C \frac{|b|^2_U}{t},
\end{align*}
\end{proposition}
\dim
We recall the idea used in \cite[Appendix]{MaPr}, which in turn is a generalization of \cite{Za} where the finite dimensional case is studied, here we also need to notice that requiring $h\in K$ is equivalent to ask that $h\in {D(A)=V\times U}$.
\newline We consider $h=\begin{pmatrix}
a \\ b
\end{pmatrix}\in V\times U$ and we set
\begin{align}\label{wave_def_contr}
\widetilde u_t(s)
=& \widetilde G(s)^{-1}\psi_1(s)
+\widetilde G(s)^{-1}\psi_2'(s) \ \in U
\end{align}
for every $s\in [0,t]$, where 
\begin{align*}
\psi_1(s)= & \Phi_t(s)
[-\Lambda^{1/2}\sin(\Lambda^{1/2}s)a+ \cos(\Lambda^{1/2}s)b], \\
\psi_2(s)= &  \Phi_t(s)[\cos(\Lambda^{1/2}s)a +\displaystyle \frac{1}{\Lambda^{1/2}}\sin(\Lambda^{1/2}s)b],
\end{align*}
and
\begin{align*}
\Phi_t(s)=\frac {s^2(t-s)^2}{\displaystyle\int_0^tr^2(t-r)^2dr} \qquad s\in[0,t].  
\end{align*}
We get that $\widetilde u_t$ is well defined in $U$ since by Hypothesis \ref{hyp:operatorG} $\widetilde G^{-1}\in L(U,U)$ and both $\psi_1(s),\psi_2'(s)\in U$.
%\newline Hence, $\widetilde u\in L^2_{\mathcal P}(\Omega\times [0,T];U)$ since $X\in L^2_{\mathcal P}(\Omega\times[0,T],H)$ and $\widetilde G^{-1}$ is Borel measurable. 
Further, we have $\Phi(0)=\Phi(t)=\Phi'(0)=\Phi'(t)=0$, $\|\Phi\|_{L^1(0,t)}=1$, $|\Phi(s)|\leq Cs^{-1}$ and $|\Phi(s)|\leq Cs^{-2}$ for every $s\in[0,t]$, for some constant $C>0$ independent of $t$. Then, we get
\begin{align*}
& \int_0^te^{(t-s)A}G(s)\widetilde u_t(s)ds  \\   
= & \int_0^te^{(t-s)A}\begin{pmatrix}
0 \\
\psi_1(s)
\end{pmatrix}ds
+\int_0^te^{(t-s)A}\begin{pmatrix}
0 \\
\psi_2'(s)
\end{pmatrix}ds \\
= & \int_0^te^{(t-s)A}\begin{pmatrix}
0 \\
\psi_1(s)
\end{pmatrix}ds
+\int_0^te^{(t-s)A}A\begin{pmatrix}
0 \\
\psi_2(s)
\end{pmatrix}ds \\
= & \int_0^te^{(t-s)A}\begin{pmatrix}
\psi_2(s) \\
\psi_1(s)
\end{pmatrix}ds 
= \int_0^te^{(t-s)A}e^{sA}h\Phi(s)ds \\
= & e^{tA}h,
\end{align*}
which gives \eqref{link_gat_der_mall_der_onde}. Here, we have used the fact that $A\begin{pmatrix}
0 \\ b
\end{pmatrix}=\begin{pmatrix}
b \\ 0
\end{pmatrix}$ and that $\begin{pmatrix}
\psi_2(s) \\ \psi_1(s)
\end{pmatrix}=e^{sA} h$ for every $s\in[0,+\infty)$. It remains to compute the $L^2$-norm of $\widetilde u_t$ in $\Omega\times (0,t)$. We have
\begin{align*}
\|\widetilde u_t\|_{L^2((0,t);U)}^2
\leq & 2\|\widetilde G^{-1}\|_{\infty}
\left(\int_0^t|\psi_1(s)|^2_{U}ds+\int_0^t|\psi_2'(s)|_{U}^2ds\right) \\
\leq & 2\widetilde C\|\widetilde G^{-1}\|_\infty\left(t^{-1}(|\Lambda^{1/2}a|_{U}^2+|b|_{U}^2)+t^{-3}(|\Lambda^{1/2}a|_{U}^2+|b|_{U}^2)\right) 
\leq C\frac{|h|_{V\times U}^2}{t^3},
\end{align*}
where $\widetilde C$ and $\widetilde M$ are positive constants which do not depend on $t\in(0,T]$, and we have used the fact that $a\in V=\Lambda^{-1/2}(U)$, that $|\Lambda^{1/2}a|_{U}=|a|_V$ and that $|h|_{V\times U}^2=|a|_V^2+|b|_{U}^2$.
\newline Finally, if we take $h=\begin{pmatrix}
0 \\ b
\end{pmatrix} \in H$ with $b\in U$, then arguing as above we get
\begin{align*}
    \|\widetilde u_t\|_{L^2((0,t);U)}^2
    \leq C \frac{|b|^2_U}{t},
\end{align*}
for some positive constant $C$ which does not depend on $t\in(0,T]$.
\qed

%%%%%%%%%%%%%%%%%%%%%%%%%

\section{The stochastic damped wave equation}
\label{sec:exa_damped}
\subsection{The general case: abstract formulation}
\label{sub:damped_general}
We consider the stochastic equation with damping given by
\begin{align*}
\left\{
\begin{array}{ll}
{\partial^2_t y}(t)=-\Lambda y(t)-\rho\Lambda^\alpha({\partial_t y}(t))+C(t,y(t),\partial_ty(t))+\sigma\left(t\right)\dot W(t), & t\in(0,T], \\[1mm]
y(0)=y_0\in D(\Lambda^{\frac12}), \quad 
\displaystyle\frac{\partial y}{\partial t}(0)=y_1\in V, &
\end{array}
\right.
\end{align*}
where $V$ is a separable Hilbert space, $\Lambda:D(\Lambda)\subset V\to V$ is a self-adjoint operator which admits an orthogonal basis of $V$ of eigenvectors $\{e_n:n\in\N\}$, with corresponding simple eigenvalues $(\mu_n)_{n\in\N}$ which blow up as $n$ tends to infinity, $C:[0,T]\times V\times V\to V$ is a bounded measurable function, Lipschitz continuous with respect to the second and the third variable, uniformly in $t\in[0,T]$, and $W$ is a $V$-cylindrical Wiener process. In abstract formulation, the above equation reads as
\begin{align*}
dX(t)=AX(t)dt+B(t,X(t))dt+G(t)dW(t), \qquad t\in[0,T], \qquad X(0)=\begin{pmatrix}
\Lambda^{\frac12}y_0 \\ y_1    
\end{pmatrix}\in H,     
\end{align*}
where $U=V$, $H=U\times U$, $X=\begin{pmatrix}
X_1 \\ X_2    
\end{pmatrix}$, $X_1=\Lambda^\frac12 y$, $X_2=\frac{\partial y}{\partial t}$ and
\begin{align*}
& A:D(A)\subseteq H\to H, \quad A=\begin{pmatrix}
0 & \Lambda^{1/2} \\[1mm]
-\Lambda^{1/2} & -\rho \Lambda^{\alpha}
\end{pmatrix},
\qquad J:U\to H, \quad Ju=\begin{pmatrix}
0 \\ u    
\end{pmatrix}, \quad u\in U, \\[1mm] 
& \bar B(t,h)=C(t,\Lambda^{-\frac12}h_1,h_2), \qquad  \ h=\begin{pmatrix}
h_1 \\ h_2    
\end{pmatrix}\in H, \ t\in[0,T], \\[1mm]
& \widetilde G:[0,T]\to \mathcal L(U), \qquad \widetilde G(t)=\sigma(t), \quad  t\in[0,T], \\[1mm]
&  G:[0,T]\to\mathcal L(U;H), \quad  G(t)=J \widetilde G(t)
= \begin{pmatrix}
0 \\ \widetilde G(t)  
\end{pmatrix}, \quad t\in[0,T].
\end{align*}

We take advantage of the spectral decomposition of $A$ described in ..., which we briefly recall here. Assume that $\mu_n$ is a simple eigenvalue of $\Lambda$ for every $n\in\N$ and that $4\mu_n^{1-2\alpha}\neq \rho^2$ for every $n\in\N$. It follows that $A$ has simple eigenvalues $\lambda_n^+$, $\lambda_n^-$, defined as
\begin{align*}
\lambda_n^{\pm}:=\frac{-\rho\mu_n^\alpha\pm\sqrt{\rho^2\mu_n^{2\alpha}-4\mu_n}}{4}, n\in\N,
\end{align*}
with corresponding normalized eigenvectors
\begin{align*}
\Phi_n^+=\begin{pmatrix}
\mu_n^\frac12 e_n \\ \lambda_n^+ e_n    
\end{pmatrix}, \qquad \Phi_n^-=\chi_n\begin{pmatrix}
\mu_n^\frac12 e_n \\ \lambda_n^- e_n    
\end{pmatrix}, \qquad n\in\N,   
\end{align*}
We decompose $H$ as $H^++H^-$ (non orthogonal, direct sum), defined as
\begin{align*}
H^+=\overline{{\rm span}\{\Phi_n^+:n\in\N\}}, \qquad H^-=\overline{{\rm span}\{\Phi_n^-:n\in\N\}}.  
\end{align*}
It follows that every $h\in H$ can be written in a unique way as $h=h^+h^-$, $h^+\in H^+$, $h^-\in H^-$, and
\begin{align*}
e^{tA} & =\sum_{n\in\N}\left(e^{t\lambda_n^+}\langle h^+,\Phi_n^+\rangle_H\Phi_n^++e^{t\lambda_n^-}\langle h^-,\Phi_n^-\rangle_H\Phi_n^-\right), \qquad h\in H, \ t\geq0, \\
Ah & =\sum_{n\in\N}\left(\lambda_n^+\langle h^+,\Phi_n^+\rangle_H\Phi_n^++\lambda_n^-\langle h^-,\Phi_n^-\rangle_H\Phi_n^-\right), \qquad h\in D(A),
\end{align*}
where $(e^{tA})_{t\geq0}$ is the semigroup generated by $A$. Further, we assume that $C$ is Lipschitz continuous with respect to $h$, uniformly with respect to $t\in[0,T]$, and that Hypothesis \ref{ip-sigma-wave} holds true.

Let us prove that, under the above assumptions and with some additional condition, Hypothesis \ref{hyp:base} are satisfied.
\begin{proposition}
\label{prop:stima_HS_1}
Assume that there exists a positive constant $\delta$ such that 
\begin{align}
\label{condizione_traccia}
\mu_n\sim  n^{\delta}, \qquad \delta>\frac{1}{\alpha}.
\end{align}
Then, Hypothesis \ref{hyp:base} are fulfilled, i.e.,  $e^{tA}G(\tau)\in\mathcal L_2(U;H)$ for every $t\in(0,T]$ and every $\tau\in[0,T]$, and there exist positive constants $C_G$ and  $\gamma_G\in\left(0,\frac12\right)$ such that \eqref{hyp_base1} is satisfied. 
\end{proposition}
\begin{remark}
Since in the case of damped wave equation $(\Lambda=-\Delta$ with Dirichlet homogeneous boundary conditions in $L^2([0,\pi]^d))$ we have $\delta=\frac 2d$, it follows that  \eqref{condizione_traccia} is satisfied if and only if $d=1$ and $\alpha>\frac 1{2}$, see Example \ref{example:damped_gen} below. 
\end{remark}
\begin{proof}
Notice that, from the assumption on $C$, condition on $\bar B$ is verified.

The proof follows from that of \cite[Proposition 5.3]{AddBig24}. Notice that, since $G(\tau)=J\widetilde G(\tau)$ and, from Hypothesis \ref{ip-sigma-wave}, $\sup_{\tau\in[0,T]}\|G(\tau)\|_{\mathcal L(U)}<\infty$, it is enough to show that $e^{tA}J\in \mathcal L_2(U;H)$. Here, it has been proved that there exists a positive constant $c$ such that for every $t\in(0,T]$ we get
\begin{align}
\label{damped_stima_norma_HS_2_1}
\left\|e^{tA}\begin{pmatrix}
0 \\ {\rm Id}    
\end{pmatrix}\right\|^2_{\mathcal L_2(U;H)}\leq c\sum_{n=1}^\infty n^{\delta(\theta-\alpha)}e^{-2\rho n^{\delta \theta}t},
\end{align}
where $\theta=1-\alpha$ if $\alpha\in\left(\frac12,1\right)$ and $\theta=\alpha$ if $\alpha\in\left[0,\frac12\right]$. Let us estimate the right-hand side of \eqref{damped_stima_norma_HS_2_1}. We get
\begin{align*}
\sum_{n=2}^\infty n^{\delta(\theta-\alpha) }e^{-2\rho n^{\delta \theta}t}
\leq \int_1^\infty x^{\delta(\theta-\alpha)}e^{-2\rho x^{\delta\theta}t}dx.
\end{align*}
By applying the change of variables $y=x^{\delta\theta}t$, we get
\begin{align*}
\int_1^\infty x^{\delta(\theta-\alpha)}e^{-2\rho x^{\delta\theta}t}dx
= & \frac{t^{-(\theta-\alpha)/\theta-1/(\delta\theta)}}{\delta\theta} \int _t^\infty y^{1/(\delta\theta)+(\theta-\alpha)/\theta-1}e^{-2\rho y}dy \\
= & \frac{t^{-(\theta-\alpha)/\theta-1/(\delta\theta)}}{\delta\theta} \int _t^T y^{1/(\delta\theta)+(\theta-\alpha)/\theta-1}e^{-2\rho y}dy \\
& + \frac{t^{-(\theta-\alpha)/\theta-1/(\delta\theta)}}{\delta\theta} \int _T^\infty y^{1/(\delta\theta)+(\theta-\alpha)/\theta-1} e^{-2\rho y}dy \\
\leq & c\left(1+t^{-(\theta-\alpha)/\theta-1/(\delta\theta)}\right).
\end{align*}
Let us notice that $-(\theta-\alpha)/\theta-1/(\delta\theta)>-1$ if and only if $\delta>1/\alpha$. It follows that under our assumptions there exists a positive constant $c$ such that
\begin{align}
\label{damped_stima_fin_HS_1_1}
\left\|e^{tA}\begin{pmatrix}
0 \\ {\rm Id}   
\end{pmatrix}\right\|^2_{\mathcal L_2(U;H)}\leq c^2t^{-2 \widetilde \gamma_G}   
\end{align}
for every $t\in(0,T]$, where $\widetilde \gamma_G=\frac12(-(\theta-\alpha)/\theta-1/(\delta\theta))$. From \eqref{damped_stima_fin_HS_1_1} and the properties of $\sigma$, we infer that
\begin{align*}
& \|e^{tA}G(\tau)\|_{\mathcal L_2(U;H)}
\leq c\|\sigma\|_\infty t^{- \widetilde \gamma_G}
\end{align*}
for every $t\in(0,T]$ and every $\tau\in[0,T]$. It follows that \eqref{hyp_base1} is fulfilled with
$C_G=c\|\sigma\|_{\infty}$ and $\gamma_G=\widetilde \gamma_G$, if $\widetilde \gamma_G>0$, and $\gamma_G$ any value in $\left(0,\frac12\right)$ if $\widetilde \gamma_G\leq 0$.
% Finally, to show that $e^{ta}G(\tau,\cdot)\in \mathcal G^1(H;\mathcal L_2(U;H))$ for every $t\in(0,T]$ and every $\tau\in[0,T]$, it is enough to notice that
% \begin{align*}
% \nabla e^{tA}G(\tau,x)y= e^{tA}\begin{pmatrix}
% 0 \\ \Lambda^{-\gamma}\nabla \widetilde G(\tau,x)y    
% \end{pmatrix}
% \end{align*}
% for every $t\in(0,T]$, every $\tau\in[0,T]$ and every $x,y\in H$. Since
% \begin{align*}
% \|\nabla G(\tau,x)y\|_U\leq \|\partial_x\sigma\|_{\infty}+\|\partial_y\sigma\|_{\infty}(\|y_1\|_U+\|y_2\|_U),    
% \end{align*}
% where $y=(y_1,y_2)$, arguing as above we conclude the proof.
\end{proof}

To prove the existence of a family of processes $(\widetilde u_t)_{t\in[0,T]}\subseteq L^2(\Omega\times[0,T];U)$ which satisfies \eqref{link_gat_der_mall_der} with $s=0$, we generalize the approach in \cite{AddMasPri23} and in \cite{AddBig24} for the case $G$ not depending on $t$. We recall that this approach relies on the idea in \cite{Za}, where the finite dimensional case is considered. Let us briefly explain such an idea. The (infinite dimensional) $2\times 2$ matrix 
\begin{align*}
[J|AJ]=\begin{pmatrix}
0 & \Lambda^{\frac12} \\ {\rm Id} & -\rho \Lambda^{\alpha}    
\end{pmatrix}
\end{align*}
is, at least formally, invertible, and its inverse matrix is formally given by
\begin{align*}
K:=[J|AJ]^{-1}
= \Lambda^{-\frac12}\begin{pmatrix}
\rho\Lambda^{\alpha}  &  \Lambda^{\frac12} \\ {\rm Id} & 0   
\end{pmatrix}
= \begin{pmatrix}
\rho\Lambda^{\alpha-\frac12}  &  {\rm Id} \\ \Lambda^{-\frac12} & 0   
\end{pmatrix}.
\end{align*}
We denote by $K_i$ the $i$-th row of $K$, with $i=1,2$, i.e., $K_1\begin{pmatrix}a \\ b\end{pmatrix}=\rho \Lambda^{\alpha-\frac12}a+b$ for all $a\in D(\Lambda^{(\alpha-\frac12)\vee 0})$ and $b\in U$, and $K_2\begin{pmatrix} a \\ b\end{pmatrix}=\Lambda^{-\frac12}a$ for every $a,b\in U$. This implies that, at least formally,
\begin{align}
\label{damped_inv_matr_op}
\begin{pmatrix}
0 \\ K_1h    
\end{pmatrix} 
+A\begin{pmatrix}
0 \\ K_2 h    
\end{pmatrix}
=JK_1h+AJK_2h
= & [J|AJ]\left[\begin{matrix}K_1 \\ K_2\end{matrix}\right]h=h, \qquad h\in H.
\end{align}

Hence, if we fix $h\in H$ and for every $t\in[0,T]$ we set
\begin{align}
\label{damped_def_ctrl}
\widetilde u_t(\tau)=\begin{cases}
\widetilde G(\tau)^
{-1}(K_1\psi_{t}(\tau)+K_2\psi_{t}'(\tau)), & \tau\in(0,t), \\
0, & \tau=0 \ \textrm{ and }\tau\in[t,T],
\end{cases}    
\end{align}
where $\psi_t(\tau)=\phi_t(\tau)e^{\tau A}h$ for every $\tau\in[0,T]$ (here, $\phi_t$ is a suitable non-negative real function which vanishes at $0$ and $t$ and with integral equal to $1$), then we get
\begin{align}
\int_0^te^{(t-\tau)A}G(\tau)\widetilde u_t(\tau)d\tau
= & \int_0^te^{(t-\tau)A}\begin{pmatrix}
0 \\ \widetilde G(\tau)\widetilde G(\tau)^{-1}(K_1\psi_{t}(\tau)+K_2\psi_{t}'(\tau))     
\end{pmatrix}d\tau \notag \\
= & \int_0^te^{(t-\tau)A}\begin{pmatrix}
0 \\ K_1\psi_t(\tau)    
\end{pmatrix} d\tau
+ \int_0^te^{(t-\tau)A}\begin{pmatrix}
0 \\ K_2\psi_t'(\tau)    
\end{pmatrix} d\tau.
\label{damped_eq_ctrl_formale_1}
\end{align}
Integrating by parts the second integral which appears in the last line of \eqref{damped_eq_ctrl_formale_1} and recalling that $\psi_t(0)=\psi_t(t)=0$, it follows that
\begin{align*}
\int_0^t e^{(t-\tau)A}\begin{pmatrix}
0 \\ K_2\psi_t'(\tau)    
\end{pmatrix} d\tau
= \int_0^t Ae^{(t-\tau)A}\begin{pmatrix}
0 \\ K_2\psi_t(\tau)    
\end{pmatrix} d\tau
= \int_0^t e^{(t-\tau)A}A\begin{pmatrix}
0 \\ K_2\psi_t(\tau)    
\end{pmatrix} d\tau.
\end{align*}
Replacing in \eqref{damped_eq_ctrl_formale_1} and taking \eqref{damped_inv_matr_op} into account, we infer that
\begin{align*}
\int_0^t e^{(t-\tau)A}G(\tau)\widetilde u_t(\tau)d\tau   = & \int_0^t e^{(t-\tau)A}\left[\begin{pmatrix}
0 \\ K_1\psi_t(\tau)    
\end{pmatrix}
+A\begin{pmatrix}
0 \\ K_2\psi_t(\tau)    
\end{pmatrix}
\right] d\tau \\
= & \int_0^te^{(t-\tau)A}\psi_t(\tau)d\tau
=  \int_0^t\phi_t(\tau)e^{(t-\tau)A}e^{\tau A}hd\tau \\
= & e^{tA}h.
\end{align*}
We stress that above computations are just formal. Indeed, to apply our reasoning we should prove that $K_1\psi_{t}(\tau)+K_2\psi_{t}'(\tau)\in D(\Lambda^{(\alpha-\frac12)\vee0})$ and
$JK_2\psi_t(\tau)\in D(A)$ for every $\tau\in(0,t)$. A detailed proof of these facts is contained in the proofs of \cite[Theorems 4.6 \& 4.7]{AddBig26} and \cite[Theorems 5.4 \& 5.6]{AddBig24}, together with an estimate of $\|\widetilde u_t\|_{L^2(\Omega\times[0,T];U)}$. The unique difference, here, is the presence of $\widetilde G(\tau)^{-1}$ in the definition of $\widetilde u$ but, due to Hypothesis \ref{ip-sigma-wave}, the estimate of $\|\widetilde u_t\|_{L^2(\Omega\times[0,T];U)}$, in terms of singularity as $t$ approaches $0^+$, does not change. We state the quoted results in the following proposition.

\begin{proposition}
\label{prop:stima_utilde}
Assume that condition \eqref{condizione_traccia} are verified, let $T>0$ be fixed, let $\alpha\in\left(0,1\right)$ and let $\widetilde u_t$ be the function defined in \eqref{damped_def_ctrl}. Hence, there exists a positive constant $C>0$ such that, for every $t\in(0,T]$ and every $a\in U$,
\begin{align}
\label{stima_ctrl_alpha>12}
\|\widetilde u_t\|_{L^2(0,T;U)}    
\leq \frac{C\|Ja\|_H}{t^{\frac12}}.
\end{align}
where $\widetilde u_t$ is defined as in \eqref{damped_def_ctrl} with $h=Ja$.
\end{proposition}

\begin{proposition}
Estimate \eqref{stima_norma_L2_utilde} is verified.
\end{proposition}
\begin{proof}
If $s=0$, then it is enough to consider $\widetilde u_{Ja,0,x,t}=\widetilde u_t$ and recall that $\|Ja\|_H=\|a\|_U$ for every $a\in U$. If $s>0$, then it suffices to notice that the above results are invariant under time translation.     
\end{proof}
Now we provide a concrete example of $\Lambda$ which satisfies our conditions. In particular, we are able to cover the case of stochastic damped wave equation in dimension $1$.
\begin{example}
\label{example:damped_gen}
Let $\Lambda$ be the realization of the Laplace operator in $L^2([0,\pi])$ with homogeneous Dirichlet boundary conditions and let $\alpha\in\left(\frac12,1\right)$. In this case, $\delta=2$ and so condition \eqref{condizione_traccia} is fulfilled. Hence, if Hypothesis \ref{ip-sigma-wave} are satisfied, then Hypothesis \ref{hyp:base} and condition \eqref{stima_norma_L2_utilde} hold true. 
\end{example}

\subsection{A special drift: the abstract formulation}
In this subsection, we consider the stochastic damped wave equation given by
\begin{align*}
\left\{
\begin{array}{ll}
\displaystyle {\partial^2_t y}(t)=-\Lambda y(t)-\rho\Lambda ^\alpha{\partial_t y}(t)+C(t,y(t))+\sigma\left(t\right)\dot W(t), & t\in(0,T], \\[2.5mm]
y(0,\xi)=y_0(\xi)\in D(\Lambda^\frac12), \quad 
\displaystyle\frac{\partial y}{\partial t}(0,\xi)=y_1(\xi)\in V,
\end{array}
\right.
\end{align*}
where $\Lambda:D(\Lambda)\subset V\to V$ and $W$ is a $V$-cylindrical Wiener process. As before, we assume that $C:[0,T]\times V\to V$ is a bounded measurable function, Lipschitz continuous with respect to the second variable, uniformly in $t\in[0,T]$. The main difference, here, is that the nonlinear term does not depend on the time derivative of $y$, and so it is possible to consider a wider space $H$ where the solution $X$ evolves. Indeed, the abstract formulation of the above equation reads as
\begin{align*}
dX(t)=AX(t)dt+B(t,X(t))dt+ GdW_\varepsilon(t), \qquad t\in[0,T], \qquad X(0)=\begin{pmatrix}
\Lambda^{-\frac12}y_0 \\ y_1    
\end{pmatrix}\in H=V_\varepsilon\times V_\varepsilon, 
\end{align*}
where $U=V_\varepsilon:=D(\Lambda^{-\varepsilon})\supset V$ ($0\leq \varepsilon\leq \frac12$ to be appropriately fixed) is the closure of $V$ with respect to the norm  $|\cdot|_{V_\varepsilon}:=|\Lambda^{-\varepsilon}\cdot|_V$, $H=U\times U$, $X=\begin{pmatrix}
X_1 \\ X_2    
\end{pmatrix}$, $y=\Lambda^{-\frac12}X_1$, $\partial_t y=X_2$, $W_\varepsilon=\Lambda^{\varepsilon}W$ is a $U$-cylindrical Wiener process and
\begin{align*}
& A:D(A)\subseteq H\to H, \quad A=\begin{pmatrix}
0 & \Lambda^{\frac12} \\[1mm]
-\Lambda^{\frac12} & -\rho \Lambda^{\alpha}
\end{pmatrix},\qquad J:U\to H, \quad Ju=\begin{pmatrix}
0 \\ \Lambda^{-\varepsilon}u    
\end{pmatrix}, \quad u\in U, \\[1mm] 
& \bar B(t,h)=C_\varepsilon(t,\Lambda^{-\frac12}h_1), \quad  C_\varepsilon(t,\Lambda^{-\frac12}h_1)=\Lambda^\varepsilon C(t,\Lambda^{-\frac12}h_1), \qquad h=\begin{pmatrix}
h_1 \\ h_2    
\end{pmatrix}\in H, \ t\in[0,T], \\[1mm] 
& \widetilde G:[0,T]\to \mathcal L(U), \quad \widetilde G(t)=\sigma(t), \qquad  t\in[0,T], \\[1mm]
&  G:[0,T]\to\mathcal L(U;H), \qquad  G(t)=J \widetilde G(t)
= \begin{pmatrix}
0 \\ \Lambda^{-\varepsilon}\widetilde G(t)  
\end{pmatrix}
= \begin{pmatrix}
0 \\ \widetilde G(t) \Lambda^{-\varepsilon}  
\end{pmatrix}, \quad t\in[0,T].
\end{align*}
We stress that, in the above definitions, $\Lambda$ is always meant as operator on $V_\varepsilon$, i.e., $\Lambda:D(\Lambda)\subset V_\varepsilon\to V^\varepsilon$.

Let us prove that, assuming Hypothesis \ref{ip-sigma-wave}, Hypothesis \ref{hyp:base} are satisfied under an extra assumption.
\begin{proposition}
\begin{enumerate}[\rm(i)]
\item There exists $C_B>0$ such that $|\bar B(x)-\bar B(y)|_U\leq C_B|x-y|_H$ for every $x,y\in H$.
\item If there exists a positive constant $\delta$ such that 
\begin{align}
\label{cond_traccia_gamma}
\mu_n\sim  n^{\delta}, \qquad \delta>\frac{1}{2\varepsilon+\alpha},
\end{align}
then, $e^{tA}G(\tau)\in\mathcal L_2(U;H)$ for every $t\in(0,T]$ and every $\tau\in[0,T]$, and there exist positive constants $C_G$ and $\varepsilon_G\in\left(0,\frac12\right)$ such that \eqref{hyp_base1} is satisfied.
\end{enumerate}
% Assume that there exists a positive constant $\delta$ such that 
% \begin{align*}
% \mu_n\sim  n^{\delta}, \qquad \delta>\frac{1}{2\varepsilon+\alpha}.
% \end{align*}
% Then, Hypothesis \ref{hyp:base} are fulfilled, i.e.,  $e^{tA}G(\tau)\in\mathcal L_2(U;H)$ for every $t\in(0,T]$ and every $\tau\in[0,T]$, and there exist positive constants $C_G$ and  $\varepsilon_G\in\left(0,\frac12\right)$ such that \eqref{hyp_base1} is satisfied.
\end{proposition}
\begin{proof}
To prove (i), we notice that, from the definition of $B$, we get, for every $x,y\in H$,
\begin{align*}
|\bar B(t,x)-\bar B(t,y)|_U
= & |C_\varepsilon(\Lambda^{-\frac12}x_1)-C_\varepsilon(\Lambda^{-\frac12}y_1)|_{V_\varepsilon}
= |C(\Lambda^{-\frac12}x_1)-C(\Lambda^{-\frac12}y_1)|_V \\
\leq & \|C\|_{\rm Lip}|\Lambda^{-\frac12}(x_1-y_1)|_V \leq \|C\|_{\rm Lip}\|\Lambda^{-\frac12+\varepsilon}\|_{\mathcal L(V)}|\Lambda^{-\varepsilon}(x_1-y_1)|_{V} \\
= & \|C\|_{\rm Lip}\|\Lambda^{-\frac12+\varepsilon}\|_{\mathcal L(V)}|x_1-y_1|_{V_\varepsilon}\leq \|C\|_{\rm Lip}\|\Lambda^{-\frac12+\varepsilon}\|_{\mathcal L(V)}|x-y|_{H}, \qquad t\in[0,T].
\end{align*}

The proof of (ii) follows the lines of that of Proposition \ref{prop:stima_HS_1}. We recall that, if $\{e_n:n\in\N\}$  is an orthonormal basis of $U$ consisting of eigenvectors of $\Lambda$ with corresponding simple eigenvalues $(\mu_n)$, the $\{\mu_n^{-\varepsilon} e_n:n\in\N\}$ is an orthonormal basis of $V_\varepsilon$. From the definition of $J$ and arguing as in the proof of \cite[Proposition 5.3]{AddBig24}, it follows that there exists a positive constant $c$ such that for every $t\in(0,T]$ we get
\begin{align}
\label{damped_stima_norma_HS_2}
\left\|e^{tA}\begin{pmatrix}
0 \\ \Lambda^{-\varepsilon}    
\end{pmatrix}\right\|^2_{\mathcal L_2(U;H)}\leq c\sum_{n=1}^\infty n^{-2\delta\varepsilon}n^{\delta(\theta-\alpha)}e^{-2\rho n^{\delta \theta}t},
\end{align}
where $\theta=1-\alpha$ if $\alpha\in\left(\frac12,1\right)$ and $\theta=\alpha$ if $\alpha\in\left[0,\frac12\right]$. From analogous computations as those in the proof of Proposition \ref{prop:stima_HS_1}, we infer that there exists a positive constant, which only depends on $T$, such that
\begin{align*}
\sum_{n=2}^\infty n^{-2\delta\varepsilon}n^{\delta(\theta-\alpha) }e^{-2\rho n^{\delta \theta}t}
\leq & c\left(1+t^{-(\theta-\alpha-2\varepsilon)/\theta-1/(\delta\theta)}\right).
\end{align*}
Let us notice that $-(\theta-\alpha-2\varepsilon)/\theta-1/(\delta\theta)>-1$ if and only if $\delta>\frac1{2\varepsilon+\alpha}$. It follows that under our assumptions there exists a positive constant $c$ such that
\begin{align}
\label{damped_stima_fin_HS_1}
\left\|e^{tA}\begin{pmatrix}
0 \\ \Lambda^{-\varepsilon}   
\end{pmatrix}\right\|^2_{\mathcal L_2(U;H)}\leq c^2t^{-2 \widetilde \varepsilon_G}   
\end{align}
for every $t\in(0,T]$, where $\widetilde \varepsilon_G=\frac12(-(\theta-\alpha-2\varepsilon)/\theta-1/(\delta\theta))$. Since $\sup_{t\in[0,T]}\|\sigma(t)\|_{\mathcal L(V_\varepsilon)}<\infty$ it follows that
\begin{align*}
& \|e^{tA}G(\tau)\|_{\mathcal L^2(U;H)}
\leq c t^{- \widetilde \varepsilon_G}
\end{align*}
for every $t\in(0,T]$, every $\tau\in[0,T]$ and some positive constant $c$. It follows that \eqref{hyp_base1} is fulfilled with
$\varepsilon_G=\widetilde \varepsilon_G$, if $\widetilde \varepsilon_G>0$, and $\varepsilon_G$ any value in $\left(0,\frac12\right)$ if $\widetilde \varepsilon_G\leq 0$.
% Finally, to show that $e^{ta}G(\tau,\cdot)\in \mathcal G^1(H;\mathcal L_2(U;H))$ for every $t\in(0,T]$ and every $\tau\in[0,T]$, it is enough to notice that
% \begin{align*}
% \nabla e^{tA}G(\tau,x)y= e^{tA}\begin{pmatrix}
% 0 \\ \Lambda^{-\varepsilon}\nabla \widetilde G(\tau,x)y    
% \end{pmatrix}
% \end{align*}
% for every $t\in(0,T]$, every $\tau\in[0,T]$ and every $x,y\in H$. Since
% \begin{align*}
% \|\nabla G(\tau,x)y\|_U\leq \|\partial_x\sigma\|_{\infty}+\|\partial_y\sigma\|_{\infty}(\|y_1\|_U+\|y_2\|_U),    
% \end{align*}
% where $y=(y_1,y_2)$, arguing as above we conclude the proof.
\end{proof}

As in the previous section, we are able to construct the function $\widetilde u_t$, by setting, for every $h\in\ H$,
\begin{align}
\label{damped_def_ctrl_2}
\widetilde u_t(\tau)=\begin{cases}
\widetilde G(\tau)^
{-1}(K_1\psi_{t}(\tau)+K_2\psi_{t}'(\tau)), & \tau\in(0,t), \\
0, & \tau=0 \ \textrm{ and }\tau\in[t,T],
\end{cases}    
\end{align}
where $\psi_t$ is the same defined in Subsection \ref{sub:damped_general} and $K_1,K_2$ are the row of the matrix-type operator $K$ given by
\begin{align*}
K:=[J|AJ]^{-1}
= \Lambda^{\varepsilon}\Lambda^{-\frac12}\begin{pmatrix}
\rho\Lambda^{\alpha}  &  \Lambda^{\frac12} \\ {\rm Id} & 0   
\end{pmatrix}
= \begin{pmatrix}
\rho\Lambda^{\alpha+\varepsilon-\frac12}  &  \Lambda^{\varepsilon} \\ \Lambda^{\varepsilon-\frac12} & 0   
\end{pmatrix},
\end{align*}
i.e., $K_1\begin{pmatrix}
a \\ b    
\end{pmatrix}=\rho \Lambda^{\alpha+\varepsilon-\frac12}a+\Lambda^{\varepsilon}b$ for all $a\in D(\Lambda^{(\alpha+\varepsilon-\frac12)\vee 0})$ and $b\in D(\Lambda^\varepsilon)$ and $K_2\begin{pmatrix}
a \\ b    
\end{pmatrix}=\Lambda^{\varepsilon-\frac12} a$ for every $a\in D(\Lambda^{\varepsilon-\frac12})$ The following result holds true.

\begin{proposition}
Assume that condition \eqref{cond_traccia_gamma} is verified and let $T>0$ and $\alpha\in(0,1)$ be fixed. Hence, there exists a positive constant $c$, which only depends on $T$, such that
\begin{align*}
%\label{stima_ctrl_alpha_caso_part}
\|\widetilde u_t\|_{L^2(0,T;U)}    
\leq \frac{C\|a\|_U}{t^\frac12}
\end{align*}
for every $t\in(0,T]$ and every $a\in U$, where $\widetilde u_t$ is defined as in \eqref{damped_def_ctrl_2} with $h=Ja$.
\end{proposition}
\begin{proof}
The proof is similar to those of [Theorems 4.6 \& 4.7]\cite{AddBig26} (see also \cite[Theorems 5.4 \& 5.6]{AddBig24}). Here, we just sketch the main modifications in the case $\alpha\in \left(0,\frac12\right]$.

in \cite[Theorem 4.6]{AddBig26}, it has been proved that
\begin{align}
K_1\psi_t(\tau)
= & -\sum_{k=1}^{+\infty}\mu_k^\varepsilon[e^{\lambda_k^+\tau }\langle h^+,\Phi_k^+\rangle_H\lambda_k^-e_k
+e^{\lambda_k^-\tau }\langle h^-,\Phi_k^- \rangle_H \chi_k\lambda_k^+e_k]
\label{forma_K_1psi_t}
\end{align}
and
\begin{align}
K_2\psi_t'(\tau )
= & \sum_{k=1}^{+\infty}\mu_k^{\varepsilon}[\phi'_t(\tau )(e^{\lambda_k^+\tau }\langle h^+,\Phi_k^+\rangle_H+e^{\lambda_k^-\tau }\langle h^-,\Phi_k^-\rangle_H\chi_k)e_k \notag \\
& +\phi_t(\tau )(\lambda_k^+e^{\lambda_k^+\tau }\langle h^+,\Phi_k^+\rangle_H+\lambda_k^-e^{\lambda_k^-\tau }\langle h^-,\Phi_k^-\rangle_H\chi_k)e_k]
\label{forma_K_2psi_t'}
\end{align}
for every $\tau \in(0,t]$. Further, for every $k\in\N$,
\begin{align}
\label{exp_Ga^+Ga^-}
\langle (Ja)^+,\Phi_k^+\rangle_H
= & \frac{-\mu_k^\varepsilon a_k}{(\lambda_k^--\lambda_k^+)\|e_k\|_V},\qquad
\langle (J_2a)^-,\Phi_k^-\rangle_H
=  \frac{-\mu_k^\varepsilon a_k}{\chi_k(\lambda_k^+-\lambda_k^-)\|e_k\|_V}.    
\end{align}
Replacing the above expressions in \eqref{forma_K_1psi_t} and in \eqref{forma_K_2psi_t'}, we infer that
\begin{align*}
K_1\psi_t(\tau)
= \phi_t(\tau)\sum_{k=1}^{+\infty}\frac{a_k}{(\lambda_k^--\lambda_k^+)\|e_k\|_V}[e^{\lambda_k^+\tau}\lambda_k^--e^{\lambda_k^-\tau}\lambda_k^+]e_k, \qquad \tau\in(0,t),
\end{align*}
and
\begin{align*}
K_2\psi_t'(\tau)
= & -\sum_{k=1}^{+\infty}\frac{a_k}{(\lambda_k^--\lambda_k^+)\|e_k\|_V}[\phi_t'(\tau)(e^{\lambda_k^+\tau}-e^{\lambda_k^-\tau})e_k+\phi_t(\tau)(\lambda_k^+ e^{\lambda_k^+\tau}-\lambda_k^-e^{\lambda_k^-\tau})e_k] \end{align*}
for every $\tau\in(0,t)$. The thesis now follows from computations in the proof of the quoted theorems, with $\gamma=0$. 
\end{proof}

As in Subsection \ref{sub:damped_general}, we recover the desired estimate.
\begin{proposition}
Estimate \eqref{stima_norma_L2_utilde} is verified.
\end{proposition}

In this case, we are able to consider stochastic damped wave equations up to dimension $3$.
\begin{example}
Let $\Lambda$ be the realization of the Laplace operator in $L^2([0,\pi]^d)$ with homogeneous Dirichlet boundary conditions with $d=1,2,3$. and assume that condition \eqref{cond_traccia_gamma} is fulfilled. In this case, $\delta=\frac2d$ and so condition \eqref{cond_traccia_gamma}, and consequently our results, are fulfilled in the following situations:
\begin{description}
\item[$d=1$] in dimension $1$, for every $\alpha\in(0,1)$ there exists $\varepsilon\in\left(0,\frac12\right)$ such that condition \eqref{cond_traccia_gamma} is verified: it is enough to consider $\varepsilon>\frac{1-2\alpha}{4}$ $($notice that, if $\alpha>\frac12$, then we can choose  $\varepsilon=0$, accordingly with Example \ref{example:damped_gen}$)$.
\item [$d=2$] In dimension $2$, condition \eqref{cond_traccia_gamma} reads as $\varepsilon>\frac{1-\alpha}{2}$, which again implies that for every $\alpha\in(0,1)$ there exists $\varepsilon\in\left(0,\frac12\right)$ such that \eqref{cond_traccia_gamma} is satisfied.
\item [$d=3$] Finally, we notice that in dimension $3$ condition \eqref{cond_traccia_gamma} is equivalent to $\varepsilon>\frac{3-2\alpha}{4}$, which implies that for every $\alpha\in\left(\frac12,1\right)$ there exists $\varepsilon\in\left(0,\frac12\right)$ such that \eqref{cond_traccia_gamma} holds true.
\end{description}

In this situation, we are also able to show that Hypothesis \ref{hyp:operatorG1} is satisfied with $J_1=\begin{pmatrix}
0 \\ {\rm Id}    
\end{pmatrix}$ and $S=\Lambda^{-\varepsilon}$, obtaining the differentialìbility of $Y_s^{s,x}$ along the directions of $\begin{pmatrix}
    0 \\ a
\end{pmatrix}$ with $a\in U$. The price to pay is a worst behaviour at $0$ of the control $\widetilde u_t$, as the following result shows.

\begin{proposition}
Assume that condition \eqref{cond_traccia_gamma} is verified and let $T>0$ and $\alpha\in(0,1)$ be fixed. Hence, there exists a positive constant $c$, which only depends on $T$, such that
\begin{align}
\label{stima_ctrl_alpha_caso_part}
\|\widetilde u_t\|_{L^2(0,T;U)}    
\leq 
\begin{cases}
\displaystyle \frac{C\|a\|_U}{t^{\frac12+\frac\varepsilon\alpha}}
, & \alpha\in\left(0,\frac12\right], \\[2mm]
\displaystyle \frac{C\|a\|_U}{t^{\frac12+\frac\varepsilon{1-\alpha}}}, & \alpha\in\left[\frac12,1\right),
\end{cases}
\end{align}
for every $t\in(0,T]$ and every $a\in U$, where $\widetilde u_t$ is defined as in \eqref{damped_def_ctrl_2} with $h=J_1a$. In particular, estimate \eqref{stima_norma_L2_utilde_1} is satisfied with $\beta=\frac\varepsilon\alpha$ and $\varepsilon<\frac\alpha2$ if $\alpha\in\left(0,\frac12\right]$ and with $\beta=\frac\varepsilon{1-\alpha}$ and $\varepsilon<\frac{1-\alpha}2$ if $\alpha\in\left[\frac12,1\right)$. 
\end{proposition}
\end{example}

%%%%%%%%%%%%%%%%%%%%%%

\end{document}